\documentclass{amsart}
\let\counterwithout\relax

\usepackage{chngcntr}
\usepackage{amsmath}
\usepackage{amssymb}
\usepackage{amsfonts}
\usepackage{graphicx}
\usepackage{marginnote}
\usepackage{color}
\usepackage{texdraw}
\usepackage{graphpap}
\usepackage{wasysym}
\usepackage{chngcntr}
\usepackage{enumitem}
\usepackage[OT2,T1]{fontenc}

\counterwithout{figure}{section}

\input txdtools

\newtheorem{mth}{Theorem}
\newtheorem*{stat}{Theorem}
\newtheorem{lemm}{Lemma}
\newtheorem{thm}{Theorem}[section]

\newtheorem{lem}[thm]{Lemma}
\newtheorem{prop}[thm]{Proposition}

\theoremstyle{definition}

\theoremstyle{remark}

\newtheorem*{defn}{Definition}
\newtheorem*{rem}{Remark}
\newtheorem*{rems}{Remarks}
\newtheorem*{ex}{Example}

\numberwithin{equation}{section}

\newcommand{\micro}{r_n}

\renewcommand{\d}{{\partial}}
\newcommand{\dbar}{{\bar{\partial}}}

\newcommand{\1}{\mathbf{1}}
\newcommand{\parat}{T}

\newcommand{\Prob}{{\mathbb P}}

\newcommand{\C}{{\mathbb C}}
\newcommand{\Int}{\operatorname{Int}}

\newcommand{\Ordo}{O}
\newcommand{\loc}{\operatorname{loc}}

\newcommand{\R}{{\mathbb R}}

\newcommand{\Z}{{\mathbb Z}}

\renewcommand{\L}{{\mathbb L}}
\newcommand{\bfR}{{\mathbf R}}

\newcommand{\Lap}{\Delta}

\newcommand{\HP}{{\mathbb{C}_+}}

\newcommand{\calF}{{\mathcal F}}

\newcommand{\eqpot}{\check{Q}}

\newcommand{\fii}{{\varphi}}
\newcommand{\config}{\Theta}

\newcommand{\bfK}{{\mathbf{K}}}

\newcommand{\Ham}{H}

\newcommand{\const}{\operatorname{const.}}

\newcommand{\re}{\operatorname{Re}}
\newcommand{\im}{\operatorname{Im}}

\newcommand{\erfc}{\operatorname{erfc}}

\newcommand{\dist}{\operatorname{dist}}
\newcommand{\supp}{\operatorname{supp}}

\newcommand{\eps}{\varepsilon}

\def\labs{\left |}
\def\rabs{\right |}
\def\babs#1{\labs {#1} \rabs}

\begin{document}
\title[Singular boundary points]
{Scaling limits of random normal matrix processes at singular boundary points}

\author{Yacin Ameur}
\author{Nam-Gyu Kang}
\author{Nikolai Makarov}
\author{Aron Wennman}
\address{Yacin Ameur\\ Department of Mathematics\\ Faculty of Science\\ Lund
University\\ P.O. Box 118\\ 221 00 Lund\\ Sweden.}
\email{yacin.ameur@maths.lth.se}
\address{Nam-Gyu Kang\\ School of Mathematics\\ Korea Institute for Advanced Study\\ 85 Hoegiro, Dongdaemun-gu, Seoul 02455\\ Republic of Korea.}
\email{namgyu@kias.re.kr}
\address{Nikolai Makarov\\ Department of Mathematics\\
California Institute of Technology\\
Pasadena\\ CA 91125\\
USA.}
\email{makarov@caltech.edu}
\address{Aron Wennman\\ School of Mathematical Sciences\\ Tel Aviv University\\ Tel Aviv 69978\\ Israel}
\email{aronwennman@tauex.tau.ac.il}

\thanks{Nikolai Makarov was supported by NSF grant no. 1500821.}



\thanks{Nam-Gyu Kang was supported by Samsung Science and Technology Foundation (SSTF-BA1401-01).}

\begin{abstract} We introduce a method for taking microscopic limits of normal matrix ensembles and apply it
to study
the behaviour near certain types of singular points on the boundary of the droplet. Our investigation includes ensembles without restrictions near the boundary,
as well as hard edge ensembles, where the eigenvalues are confined to the droplet.
We establish in both cases existence of new types of determinantal point fields, which differ from those which
can appear at a regular boundary point, or in the bulk.
\end{abstract}

\maketitle

The method of rescaled Ward identities was introduced in \cite{AKM},
where the main focus was on scaling limits of normal matrix eigenvalue ensembles near a regular boundary point of the droplet.
In this note, we apply similar methods and study scaling limits near
\textit{singular} boundary points, which may be cusps, double points, crossing points, and possibly other types.

\begin{figure}[ht]
\begin{center}
\includegraphics[width=.2\textwidth]{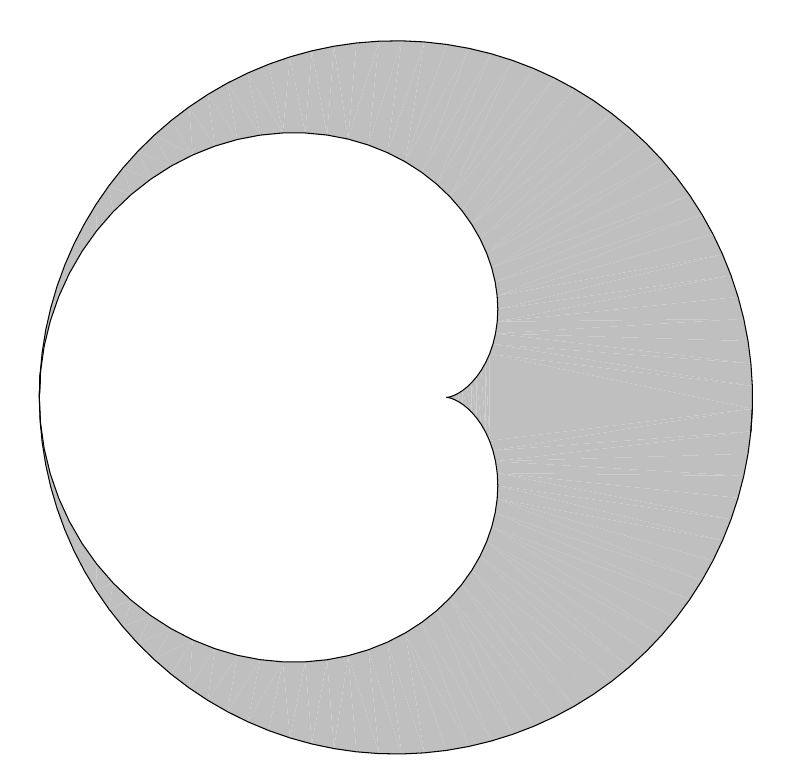}
\hspace{.05\textwidth}
\includegraphics[width=.2\textwidth]{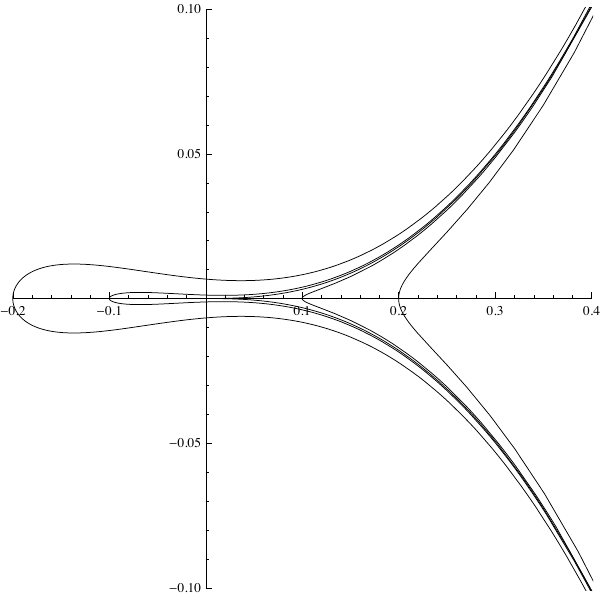}
\end{center}
\caption{The figure on the left shows a boundary with two singular points: one double point and one cusp.
The figure on the right shows a $(5,2)$-cusp embedded in a Hele-Shaw flow of boundaries of droplets.}
\label{fig1}
\end{figure}

In the normal matrix model, we start with a suitable real-valued function $Q$,
which we call the potential. We consider random configurations (or systems) $\{\zeta_j\}_1^n$ of points in $\C$,
having the interpretation of identical point charges subject to the external field
$nQ$. The system is picked with respect to the probability measure $\Prob_n$ on $\C^n$ given by
\begin{equation}\label{E1.1}
d \Prob_n=\frac 1 {Z_n}\,e^{\,-\,\Ham_n}\,d V_n,
\qquad \Ham_n
:=\sum_{j\ne k}^n\log\frac 1 {\babs{\,\zeta_j-\zeta_k\,}}
+n\sum_{j=1}^nQ\left(\zeta_j\right).\end{equation}
Here $dV_n$ denotes Lebesgue measure in $\C^n$ divided by $\pi^n$ and
$Z_n$ is chosen so that $\Prob_n(\C^n)=1$.

In the thermodynamic limit $n\to\infty$, the
system tends to condensate on a compact set $S$ known as the
droplet, the boundary of which is a finite union of real-analytic arcs, possibly containing finitely many singular
points where the arcs meet. We shall here investigate the density of eigenvalues near singular boundary points and in particular establish
existence of new kinds of scaling limits (determinantal point-fields) which emerge by zooming appropriately near the given singular point.

We shall mainly study singular boundary points where the decisive condition $\Lap Q>0$ is satisfied; such points
are either cusps (of certain types) or double points.
For such points,
we shall find nontrivial scaling limits located somewhat inside the droplet,
by zooming about a moving location, which approaches the
singular point at a proper rate, cf. Figure~\ref{fig: moving}. Another type of singularity, a {\em crossing point},
may emerge at a boundary point where $\Lap Q=0$, as in the example of the {\em lemniscate ensemble} \cite{BEG}.
In this case we zoom at the singular point itself, but due to the vanishing of the equilibrium density,
we require a relatively coarse scale in order to recover a nontrivial scaling limit.

In addition, we shall consider scaling limits under {\em hard
edge boundary conditions} (or {\em hard edge confinement}), leading to yet other families of determinantal point fields.

\begin{figure}[ht]
\begin{center}
\includegraphics[height=.2\textwidth]{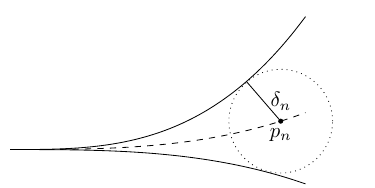}
\includegraphics[height=.18\textwidth]{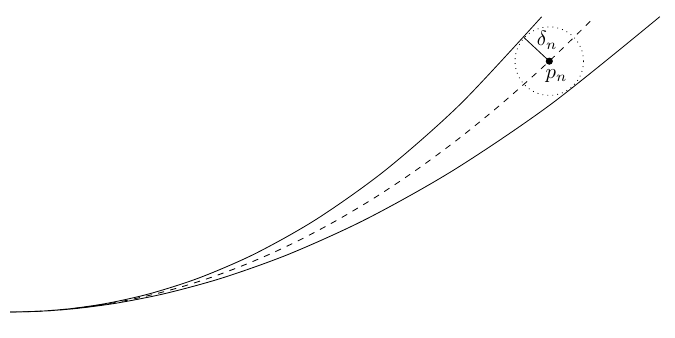}
\end{center}
\caption{A moving point $p_n$ on the (dashed) {\em bisectrix} approaching a cusp of type $(\nu,2)$ with $\nu$ odd (left) and $\nu$ even (right).
To obtain non-trivial scaling limits, we choose $p_n$ so that $\delta_n=T/\sqrt{n\Lap Q}$ where $T>0$ is a parameter.}
\label{fig: moving}
\end{figure}

\subsection*{Notational conventions} We write $D(p;r)$ for the open disc centered at $p$
of radius $r$, $\HP$ for the open upper half-plane $\{\im \zeta>0\}$, and $\C^*$ for the punctured plane
$\C\setminus\{0\}$. If $E$ is a subset of $\C$ we write $E^c=\C\setminus E$ for the complement, while the symbol $\1_E$
will stand for the characteristic function of $E$. We use the notation $\Lap=\d\dbar$, so $\Lap$ is $1/4$ of the usual Laplacian.
We write $dA(z)=-(2\pi i)^{-1}dz\wedge d\bar{z}$ for Lebesgue measure in $\C$,
normalized so that the unit disc has unit area.
A continuous function $f:\C^2\to\C$ is termed \textit{Hermitian} if $f(z,w)=\overline{f(w,z)}$.
We say that $f$ is \textit{Hermitian-analytic} (or \textit{-entire}) if $f$ is Hermitian and analytic (entire) as a function of $z$ and
$\bar{w}$.
A Hermitian function $c$ is called a \textit{cocycle} if
$c(z,w)=g(z)\overline{g(w)}$ for a continuous unimodular function $g$.
We write $\operatorname{Pol}(k)$ for the linear space of analytic polynomials of degree at most $k$.
The symbol $a_n\sim b_n$ denotes that $b_n/a_n\to 1$ as $n\to\infty$, where $a_n,b_n$ are positive numbers.
We use the notation $a_n\gtrsim b_n$ if there exists some constant $c>0$ such that
$a_n\ge c \,b_n$ for all large $n$.
The notions $a_n\lesssim b_n$ and $a_n\asymp b_n$ are defined
analogously.

\section{Introduction and main results}

\subsection{External potential and droplet} Our basic setup is as in \cite{AKM}.
We let $Q:\C\to \R\cup\{\infty\}$ be a suitable, lower semicontinuous, external potential of sufficient growth,
$$
\liminf_{\zeta\to\infty}\frac{Q(\zeta)}{\log|\,\zeta\,|^{\,2}}>1.
$$
We assume also that $Q$ be finite on some open set.

If $\mu$ is a positive, compactly supported Borel measure we define its logarithmic $Q$-energy by
\begin{equation}I_Q[\mu]=\int_\C Q\, d\mu+\int_{\C^2}\log\frac 1 {\babs{\,\zeta-\eta \,}}\,
d\mu(\zeta)\,d\mu(\eta).\end{equation}

We will use the following basic facts of weighted potential theory, cf. \cite{ST}.

Under the conditions given above, there exists a unique \textit{equilibrium measure}
$\sigma$ of unit mass which minimizes $I_Q[\mu]$ over all compactly supported Borel
probability measures $\mu$ on $\C$. The compact support of the measure $\sigma$ is called the
\textit{droplet} in the external field $Q$, and is denoted
$$
S=S[Q]:=\supp\sigma.
$$
If $Q$ is smooth in some neighbourhood of $S$, the measure $\sigma$ is absolutely continuous
and of the form
\begin{equation}\label{sform}
d\sigma(z)=\Lap Q(z)\,\1_S(z)\, dA(z).
\end{equation}
Indeed, we assume in the following that $Q$ is {\em real-analytic} in some neighbourhood of
the boundary $\d S$.

 Note that $\Lap Q\ge 0$ on $S$, since $\Lap Q\cdot \1_S$ has
the meaning of the density of a positive measure.

We will consider boundary points $p\in\d S$ which fall in two categories.

\smallskip

- $p$ is said to be an {\em ordinary} boundary point if $\Lap Q(p)>0$.

- $p$ is said to be a {\em special} boundary point if $\Lap Q(p)=0$.

\smallskip

Ordinary boundary points have been classified by Sakai \cite{HS,Sa}, providing a suitable platform
to study such points in complete generality. To our knowledge, there does not seem to exist a similar classification of special boundary points, and we shall merely
compare with some examples of such points, which emerge naturally
in the recent papers \cite{BEG,GPSS}.

It is convenient here to briefly overview the elements of Sakai's theory which are relevant for our present investigation.

\subsection{Sakai's theory} The famous regularity theorem of Sakai states that if a domain $\Omega\subset\C$ has a (local) Schwarz function at $p\in \d \Omega$, then $p$ is either a regular point, a double point, or a \textit{conformal} cusp of the boundary $\d \Omega$. Conformal cusps can be classified according to degree of tangency, into classes of "$(\nu,2)$-cusps" for $\nu\ge 3$. (See \cite{LM,Sa,Sh}.)

Another important result of Sakai concerns regularity of free boundaries in obstacle problems: if $u\in C^1(\Omega)$, $u\ge 0$, $\Lap u=1$ in $\Omega:=\{u>0\}$ and $0\in\d \Omega$ then $\Omega$ has a local Schwarz function at $0$, and if it is a cusp point then $\nu\not\equiv 3\, (\operatorname{mod} 4)$.
See \cite{Sa2}.

In particular, there are no $(3,2)$-cusps on a free boundary. There are several versions of this statement, e.g. maximality of $(3,2)$-cusps for Hele-Shaw flows. Figures \ref{fig1} and \ref{fig7} illustrate that $(5,2)$-cusps can in fact appear on a free boundary.

A \textit{local} droplet of $Q$ is more general than an ordinary droplet; this kind of droplet is natural in connection with hard edge theory.
If $K$ is a local droplet and $Q$ is real-analytic in a neighbourhood of $\d K$ then $\Omega=K^c$ has a Schwarz function at all boundary points, so Sakai's regularity theorem can be applied. However, we are this time not dealing with a free boundary, and e.g. the deltoid in Fig. \ref{fig: deltoid} is a local (and maximal) droplet which has $(3,2)$-cusps on its boundary.

It is believed that Sakai's classification theorem for cusps should hold if one replaces "$1$" in $\Lap u=1$ by any positive real analytic function, but we are not aware of a detailed proof of this in the literature. A consequence is that if $K$ is a droplet (rather than just a local droplet) of $Q$, and if $Q$ is real-analytic in a neighbourhood of $\d K$, then $(3,2)$-cusps do not appear on the boundary of $K$.

Since this last fact (exclusion of $(3,2)$-cusps) will be crucial in our analysis of free boundary droplets, we will include a self-contained proof, see
Proposition \ref{32prop} below.

\subsection{Ward's equation} We now briefly describe our main tool of Ward identities (or loop equations).
This kind of exact identities are well-known in 1-dimensional random matrix theories and have for example been used to prove Gaussian
field convergence for suitable linear statistics (Johansson's theorem). On the other hand, powerful techniques such as Riemann-Hilbert methods have been successfully applied to study many problems concerning eigenvalue spacing in dimension 1.

A new feature of Ward identities in dimension 2 appears in the present work as well as in the earlier
companion paper \cite{AKM}: if we rescale two-dimensional Ward identities at a natural microscopic scale, somewhat surprisingly we
avoid blow up and obtain equations.
This essentially reduces the question
of establishing universality to a problem of supplying an equation with appropriate side conditions to guarantee uniqueness of
a solution. These side conditions necessarily depend
on the nature of the point we are zooming at, and in particular on its position relative to the droplet.

In the case of scaling limits near singular boundary points, a natural side condition is known as \textit{translation invariance}.
We shall here study point fields with this property by exploiting
the fact that translation invariant solutions to Ward's equation were completely classified in \cite{AKM}.

\subsection{Ordinary boundary points} We now turn to a more detailed description of ordinary boundary points, i.e., points at which $\Lap Q>0$.

The most common type of boundary point is a \textit{regular} point. This is a point $p$
such that there exists a neighbourhood $D=D(p;\epsilon)$ such that $D\setminus S$ is a Jordan domain
and $D\cap(\d S)$ is a simple real-analytic arc. By Sakai's regularity theorem, all but finitely
many boundary points of $S$ are regular. The finitely many exceptional points are called \textit{singular}.

When analyzing a singular point, we can without loss of generality assume that it is located on
the \textit{outer} boundary of $S$, i.e., on the boundary of the unbounded component $U$ of
$\hat{\C}\setminus S$. If there are other boundary components, they can be treated in the
same way.

There are two kinds of ordinary singular boundary points.

A point $p\in\d U$ is called
a (conformal) \textit{cusp} if there is $D=D(p;\epsilon)$ such that $D\setminus S$ is a Jordan domain and
every conformal map $\Phi:\HP\to D\setminus U$ with $\Phi(0)=p$ extends analytically to a neighbourhood
of $0$ and satisfies $\Phi'(0)=0$. We remark that the cusps which appear at an ordinary boundary point of $S$
point {\em out} of $S$. (This follows since the complement $S^c$ is a generalized quadrature domain, see e.g. \cite{Sh0,Sh}.)

The second possibility is that $p$ is a \textit{double point}, i.e., that there is a small enough disc $D$ about $p$ such that
$D\setminus S$ is a union of two Jordan domains, and $p$ is a regular boundary point of each of them.

One can further classify singular points according to degrees of tangency. We shall now briefly recall how this works for cusps.

Assume that $\d S$ has a cusp at the outer boundary $\d U$ at $p=0$. We can assume that a conformal
map $\Phi:\HP\to U$ satisfies $\Phi(0)=0$ and
$$\Phi'(z)=z+a_2z^2+\cdots
+(a_{\nu-1}+ib)z^{\nu-1}+\cdots$$
where $a_j$ and $b$ are real and $b\ne 0$. Then
$$\Phi(z)=\frac 1 2 z^2+\frac {a_2} 3 z^3+\cdots
+\frac {a_{\nu-1}+ib}
{\nu}z^{\nu}+\cdots.$$
If we write $\Phi=u+iv$, we find
\begin{equation}\label{wefind}u(x)=\frac 1 2x^2+\cdots,\quad v(x)=\frac b {\nu}x^{\nu}+\cdots,\quad (x\in\R).\end{equation}
By definition, this means that the cusp at $0$ is of type $(\nu, 2)$.

Some cusps, in particular $(3,2)$-cusps, which are generic in Sakai's theory
for boundaries admitting a Schwarz function, can not appear on a free boundary, at least not at an ordinary point.
However, $(\nu,2)$ cusps for $\nu\ge 4$, $\nu\not\equiv 3 \mod 4$ do appear, and are treated below.

Droplets with ordinary singular boundary points have been studied in the papers
\cite{BBLM,BK,LM,Sa2,T}, the book \cite{Sa3}, and in the thesis \cite{C}.

\subsection{Rescaled ensembles and limiting point-fields}
Let $\{\zeta_j\}_1^n$ be a random sample from the
distribution \eqref{E1.1}.
As in \cite{AKM}, we will denote objects pertinent to this ("non-rescaled") ensemble
by
boldface characters. For example, we denote the $k$-point function by the symbol $\bfR_{n,k}$; this is defined for distinct $\eta_1,\ldots,\eta_k$, by
$$
\bfR_{n,k}\left(\eta_1,\ldots,\eta_k\right)=\lim_{\eps\downarrow 0}\frac1{\eps^{2k}}\,
\Prob_n\big(\{\textrm{There is at least one particle in each disc } N_{D( \eta_j;\eps)}\}\big).
$$

It is well-known that the process with law \eqref{E1.1} is determinantal, i.e., that we have
$$\bfR_{n,k}(\zeta_1,\ldots,\zeta_k)=\det(\bfK_n(\zeta_i,\zeta_j))_{i,j=1}^k,\quad
$$
where $\bfK_n$ is a Hermitian function, which we call a correlation kernel of the process.
More precisely, a correlation kernel $\bfK_n$ may be obtained as the reproducing kernel for the space
of weighted polynomials
$$
w(\zeta)=f(\zeta)e^{-nQ(\zeta)/2},\qquad f\in \operatorname{Pol}(n-1)
$$
endowed with the topology of
$L^2(\mathbb{C},dA)$. (Cf. \cite[Ch. IV.7.2]{ST} or \cite{M} for proofs.)

Now consider a sequence of points $p_n\in S$. We define the {\em microscopic scale} $r_n$
at
$p_n$ to be the smallest number $r_*>0$ such that
\begin{equation}\label{rlap}n\int_{D(p_n;r_*)}\Lap Q\, dA=1.\end{equation} It is easy to see that if $p_n\to p$ where
$\Lap Q(p)>0$, then $r_n\sim 1/\sqrt{n\Lap Q(p)}$ as $n\to\infty$.

\smallskip

In the following, we shall often exploit the freedom to choose an $n$-dependent coordinate system
so that the point $p_n$ remains at the origin for all $n$: $p_n=0$. We are also free to rotate our
coordinate system so that a given direction coincides with, say, the positive real axis.
This is the
\textit{passive} interpretation.
In some instances, we shall prefer to use the \textit{active} interpretation,
where the coordinate system is static while $p_n$ moves around.

\smallskip

Consider now the passive picture $p_n=0$ and define a rescaled point processes $\{z_j\}_1^n$ by
\begin{equation}\label{rescale}
z_j=\micro^{-1}\zeta_j.
\end{equation}
Objects pertaining to the rescaled system $\{z_j\}_1^n$ are denoted by plain symbols.
For example, the $k$-point function of the rescaled system will be written
$$R_{n,k}(z_1,\ldots,z_k):= {\micro^{2k}}\,\bfR_{n,k}(\zeta_1,\ldots,\zeta_k).$$

The rescaled process is likewise determinantal, but having the rescaled correlation kernel
$$K_n(z,w)=\micro^{2}\, \bfK_n(\zeta,\eta),\qquad z=\micro^{-1}\zeta,\qquad  w=\micro^{-1}\eta.$$

Recall that we allow the possibility that $Q=+\infty$ on some portion of $\C$. In general, we define the {\em one particle space} of the process $\{z_j\}_1^n$ to be the
set
$$E=\{Q<+\infty\}.$$

Our goal is to discuss non-trivial limiting point fields $\{z_i\}_1^\infty$ which are subsequential limits of the finite
processes $\{z_i\}_1^n$, along some subsequence $(n_l)_{l=1}^\infty$ of the positive integers.
The precise meaning of such a convergence is that,
for each fixed $k$, we have convergence in $L^1_{\loc}(E^k)$
$$R_{n_l,k}\to \rho_k,\qquad  (l\to\infty),$$
where $\rho_k$ is a function on $\C^k$ (namely the $k$-point function of a limiting point field).

Below we write $R_n=R_{n,1}$ for the $1$-point function of $\{z_j\}_1^n$ and
$R=\rho_1$ for the $1$-point function
of a limiting point field $\{z_j\}_1^\infty$.

\begin{lemm}\label{TH0}
Suppose that $R_{n_l}\to R$ in $L^1_{\loc}(E)$ as $l\to\infty$.
Then there exists a unique determinantal point field $\{z_j\}_1^\infty$ in $E$ with one-point
function $R$, such that $\{z_j\}_1^{n_l}\to \{z_j\}_1^\infty$ in the sense of point fields.
\end{lemm}

\begin{proof} Convergence in the sense of point fields, as well as existence and uniqueness of a scaling limit,
follows from Lenard's theory \cite{L1}-\cite{L3}.
Alternatively, we can
use the
Macchi-Soshnikov theorem (see \cite{S}), since it will be seen below that $R(z)=K(z,z)$
where $K$ is a locally trace class projection kernel.
\end{proof}

\subsection{Special functions} We here list a number of special functions and operations that are
used throughout the paper.

Denote the correlation
kernel of the infinite Ginibre ensemble by
$$
G(z,w):=e^{z\bar{w}-|z|^2/2-|w|^2/2}.
$$
By the \textit{Gaussian kernel} $\gamma$ we mean the entire function
$$
\gamma(z)=\frac 1 {\sqrt{2\pi}}e^{-z^2/2}.
$$
If $\fii$ is a suitable "window function" (tempered distribution) on $\R$, we define the convolution
with $\gamma$ to be the entire function
$$
\Phi(z)=\gamma*\fii(z)=\int_{-\infty}^{+\infty}\fii(t)\gamma(z-t)\, dt.
$$

In particular, choosing $\fii=\1_{(-\infty,0)}$, we obtain the
{\em free boundary function} $F$, which can also be expressed in terms of the complementary error function as follows
\begin{equation}\label{freb}F(z):=\gamma*\1_{(-\infty,0)}(z)=\frac 1 2\erfc\Big(\frac z {\sqrt{2}}\Big).\end{equation}

\subsection{Limiting kernels near ordinary boundary points}
Let us consider first the case of a moving origin $0=0_n\in S$ such that the condition
\begin{equation}\label{for}\liminf_{n\to\infty} \Lap Q(0)\ge \const>0\end{equation}
is satisfied. We consider for the
time being any such moving origin, not necessarily converging to a point of $S$, but having
of course subsequential limits in $S$, at which $\Lap Q>0$.

\smallskip

By definition (see \eqref{rlap}) the microscopic scale
is $r_n=(1+o(1))/\sqrt{n\Lap Q(0)}$. By a slight abuse of notation, we will
neglect the $o(1)$ term here, writing $r_n=1/\sqrt{n\Lap Q(0)}$. Thus (in {\em all} cases when \eqref{for} holds) we shall
rescale by
$$
z_j=\sqrt{n\Lap Q(0)}\,\zeta_j,\qquad j=1,\ldots,n.
$$

\begin{lemm} \label{l0}
\renewcommand{\theenumi}{\roman{enumi}}
\begin{enumerate}[leftmargin=0.27in]
\item \label{l0i} There exists a sequence of cocycles $c_n$ such that every subsequence of the functions $(c_nK_n)$ has a
further subsequence converging locally uniformly on $\C^2$ to a limit $K=G\Psi$, where
$\Psi$ is a Hermitian entire function satisfying the "mass-one inequality"
\begin{equation}\label{moi}\int_\C e^{-|z-w|^2}\babs{\Psi(z,w)}^2\, dA(w)\le \Psi(z,z),\qquad (z\in\C).\end{equation}
\item \label{l0ii} The function $R(z):=K(z,z)=\Psi(z,z)$ is either trivial (i.e., identically zero), or else it is everywhere
strictly positive. Moreover, $R\le 1$ everywhere.
\item \label{l0iii} If $R$ is non-trivial, then Ward's equation holds pointwise on $\C$
\begin{equation}\label{ward}\dbar C=R-1-\Delta\log R,\end{equation}
where
$$C(z):=\int_\C \frac{B(z,w)}{z-w}\, dA(w), \qquad B(z,w):=\frac{\babs{K(z,w)}^2}{R(z)}.$$
\item \label{l0iv} If the moving origin $0_n$ is in the "bulk regime" in the sense that $\sqrt{n}\cdot \dist(0_n,\d S)\to\infty$ as $n\to\infty$, then $R\equiv 1$.

\end{enumerate}
\end{lemm}

\begin{proof}[Remark on the proof] \eqref{l0i}-\eqref{l0iii} follow from Theorems 1.1-1.3 in \cite{AKM}, if we just
observe that the normal families argument in \cite[Section 3]{AKM} works equally fine when we rescale
about an $n$-dependent point $p=0_n$, provided that the decisive condition \eqref{for} holds.
If further $\sqrt{n}\dist(0_n,\d S)\to\infty$, then $R\equiv 1$ by the a priori estimates in \cite[Section 5]{AKM}.
\end{proof}

A limit point $K$ in Theorem~\ref{l0} will be  called a \textit{limiting kernel}, and
$R$ is the corresponding \textit{limiting $1$-point function}.

\smallskip

There is nothing which prevents a limiting kernel $K$ from being trivial, i.e., we may well have $K=0$.
In this case the limit is the trivial point field, all of whose $k$-point
functions vanishes identically. On the other hand, the case when we rescale about a regular boundary point
was recently settled in a fairly general situation:

\begin{stat} (\cite{HW})
Let $p$ be a point on the outer boundary of $S$ and suppose that we rescale in the outwards normal direction.
Suppose also that all points on the outer boundary of $S$ be regular. Then $R(z)=F(z+\bar{z})$ where $F$ is the function in \eqref{freb}.
\end{stat}

By contrast, at an ordinary singular boundary point, we have the following "triviality theorem".

\begin{mth} \label{THI} Let $p\in \d S$ be an ordinary singular boundary point and rescale about $p$ by
\begin{equation}\label{thIS}
z_j=\sqrt{n\Lap Q(p)}(\zeta_j-p),\quad j=1,\ldots,n.
\end{equation}
Then any limiting $1$-point function $R$ vanishes identically.
\end{mth}

To obtain nontrivial two-dimensional scaling limits (using the scaling \eqref{thIS})
we shall rescale about a moving point located slightly inside the droplet, as follows.

\begin{defn}
Let $p\in \d S$ be an ordinary singular boundary point and fix a positive parameter $\parat$.

\begin{enumerate}[label=(\roman*)]
\item
If $\d S$ has a cusp at $p$, we consider the point $0_n\in S$ of distance
$\delta_n=\delta_n(T):=\parat/\sqrt{n\Lap Q(p)}$ from the boundary $\d S$, which is closest to the singular
point $p$. (See Fig. \ref{fig: moving}.)
\item
If $S$ has a double point, there are instead two distinct points $0_n'$, $0_n''$ in $S$ of distance
$\delta_n$
to $\d S$, of minimal distance to $p$. We let $0_n$ denote one of these two points.
\end{enumerate}
\end{defn}

Suppose first that $S$ has a $(\nu,2)$-cusp at $p$.
Rescale about $0=0_n$ according to
\begin{equation}\label{mop}z_j=\sqrt{n\Lap Q(0)}\,\zeta_j,\qquad j=1,\ldots,n,\end{equation}
where the axes of the $z$-plane are chosen so that the imaginary axis is tangent to the bisectrix, the positive imaginary direction pointing
"towards" the cusp, cf. Figure \ref{fig: moving}.

\smallskip

We emphasize that, as $n\to\infty$, the droplet looks (locally) more and more like the strip
\begin{equation}\label{Estrip}
\Sigma_\parat=\{z;\, -\parat\le \re z\le \parat\}.
\end{equation}

\begin{mth}\label{THII}
If $\parat$ is sufficiently large, then each limiting $1$-point function $R(z)=K(z,z)$ is everywhere
positive and gives rise to a solution to Ward's equation. Moreover,
$R$ satisfies the estimate
\begin{equation}\label{damp}
R(z)\le Ce^{-2(|x|-\parat)^2},\qquad (x=\re z).
\end{equation}
\end{mth}

\begin{mth}\label{THIII} If $S$ has a double point at $p$, we rescale as in \eqref{mop} with $0_n$
equal to either $0_n'$ or $0_n''$. The conclusions of Theorem \ref{THII} then hold also for the
limiting $1$-point function $R$ about $0_n$.
\end{mth}

\begin{rems} (i) The assumption in Theorem \ref{THII} that the parameter $\parat>0$ be sufficiently large is made for technical reasons of the
proof. We do not think it should be necessary. This notwithstanding, we remark that the estimate \eqref{damp} is always true, for all $\parat>0$, as our proof
below shows.

(ii) The limiting point fields, whose existence is guaranteed by Theorems
\ref{THII} and \ref{THIII} are necessarily different from those
which can appear at a fixed regular boundary point. Indeed, as was observed in \cite{AKM},
a limiting $1$-point function rescaled in the outer normal direction about
a regular boundary point will satisfy the estimate
$
\babs{R(z)-\1_{(-\infty,0)}(x)}\le Ce^{-cx^2}
$
where $c$ is some positive constant. This estimate is clearly not consistent with \eqref{damp}.

(iii) It is interesting to compare with results  in the weakly Hermitian case, where the droplet is a narrow ellipse of height
proportional to $1/n$. (This is investigated in the papers \cite{FKS,ACV} and references.) In the "bulk", nontrivial scaling limits emerge at the $1/n$-scale, rather than at $r_n\propto 1/\sqrt{n}$. The relationship to our present setting will be clarified in a separate publication.
\end{rems}

We now discuss a family of natural candidates for limiting point fields in the above setting.

By definition, a point field with $1$-point function $R$ is called \textit{(vertical) translation invariant}
if $R(z)=R(x)$ where $x=\re z$.

Considering that the limiting droplet is the translation invariant strip \eqref{Estrip},
it seems highly plausible that each limiting 1-point function $R$ should be translation invariant as well.

In any case, we shall now use theory from \cite{AKM}
to narrow down the set of possible limiting kernels, under the extra hypothesis of translation invariance.

\begin{mth} \label{THIV}
The only non-trivial translation invariant solutions $R$ to Ward's equation~\eqref{ward} satisfying
the estimate~\eqref{damp} are given by
\begin{equation}\label{ER}
R(z) = F_I(2\re z),
\end{equation}
where $F_I$ is an entire function of the form
$$F_I(z):=\frac1{\sqrt{2\pi}} \int_I e^{-(z-t)^2/2}\,dt$$
for some interval $I$ contained in $[-2\parat,2\parat].$
\end{mth}

\begin{figure}[ht]\label{fig3}
\begin{center}
\includegraphics[width=.99\textwidth]{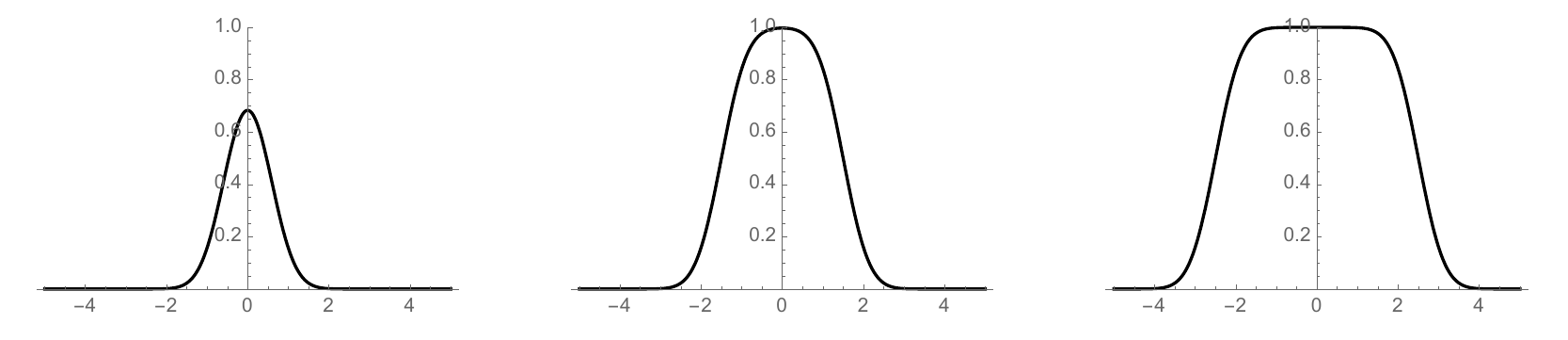}
\end{center}
\caption{The function
$x\mapsto F_{[-2\parat,2\parat]}(2x)$ for $\parat=1/2$, $\parat=3/2$, and $\parat=5/2$.}
\end{figure}

The limit $R$ in \eqref{ER} corresponds to the locally trace class projection kernel $K_I$ on $L^2(\C)$
given by $K_I(z,w)=e^{z\bar{w}-|z|^2/2-|w|^2/2}F_I(z+\bar{w})$, in the sense that $R(z)=K_I(z,z)$. (See \cite{AKM}, in particular Section 8.2, for more about this relationship.)
Hence each such $R$
determines a unique determinantal point field by the Macchi-Soshnikov theorem.

\smallskip

We {\em conjecture} that each limiting point field in Theorem \ref{THII} and Theorem \ref{THIII} is translation invariant, and thus that \eqref{ER} should give a complete list of scaling limits.
More precisely, we conjecture that the full interval $I=[-2\parat,2\parat]$ will appear in Theorem \ref{THIV}.

\smallskip

Observe that the 1-point functions $R_{[-2\parat,2\parat]}$ interpolates in a natural way between $R_\emptyset=0$ (at the singular point) and the
Ginibre kernel $R_\R=1$ (the bulk). See Figure 3.

\subsection{Lemniscate ensembles} To exemplify special boundary points,
we will now take a brief look at the potential
\begin{equation}\label{lemp}
Q=Q_k=|\zeta|^{2k}-2k^{-1/2}\re(\zeta^k).
\end{equation}
where $k\ge 2$ is an integer. Somewhat more generally, we will consider the $n$-dependent potential
$$
V_n(\zeta)=Q(\zeta)+\frac {2c} n\log\frac 1 {|\zeta|}
$$
where $c>-1$. This potential corresponds to insertion of a charge of strength $c$ at the origin relative to the external field $nQ$, see \cite{AKS}.

It is known (see \cite{BEG,GPSS} and references) that the droplet $S$ corresponding to $Q$ is the interior of the lemniscate
$|\zeta^k-1/\sqrt{k}|=1/\sqrt{k}$, while the equilibrium measure is
given by the density
${k^2}|\zeta|^{2k-2}\1_S(\zeta).$ In particular, $0\in\d S$ and $\Lap Q(0)=0$, so the origin is a special singular
boundary point.

A natural rescaling is
\begin{equation}\label{coarse}
z=r_n^{-1}\zeta,\qquad\qquad  r_n:=n^{-1/2k}.
\end{equation}
We write $K_n(z,w)=r_n^2\bfK_n(\zeta,\eta)$ for the rescaled kernel and put $V_0(z)=|z|^{2k}-2c\log|z|$.
Also write $d\mu_0(z)=e^{-V_0(z)}\, dA(z)$ and let $L^2_a(\mu_0)$ be the corresponding Bergman
space of entire functions. The Bergman kernel in this space is denoted $L_0(z,w)$.

The following compactness result is a special case of \cite[Theorem 1.1]{AKS}.

\begin{lemm} \label{th1}
There exists a sequence of cocycles $c_n$ such that
$$c_n(z,w)K_n(z,w)=L_n(z,w)e^{-V_0(z)/2-V_0(w)/2}(1+o(1))$$
where $L_n$ is Hermitian-entire and $o(1)\to 0$ locally uniformly on $\C^2$.
Moreover, each subsequence of the $L_n$'s has a further subsequence converging locally
uniformly on $\C^2$ to a Hermitian-entire limit $L$ which satisfies $L\le L_0$ in the sense of positive matrices.
\end{lemm}

Notice that, after the rescaling \eqref{coarse}, in the limit $n\to\infty$, the droplet takes the form of the "star"
$\Sigma=\{z;\, \re z^k\ge 0\}$ (see Fig. \ref{FIGURE4}).

\begin{figure}[ht]
\begin{center}
\includegraphics[width=.25\textwidth]{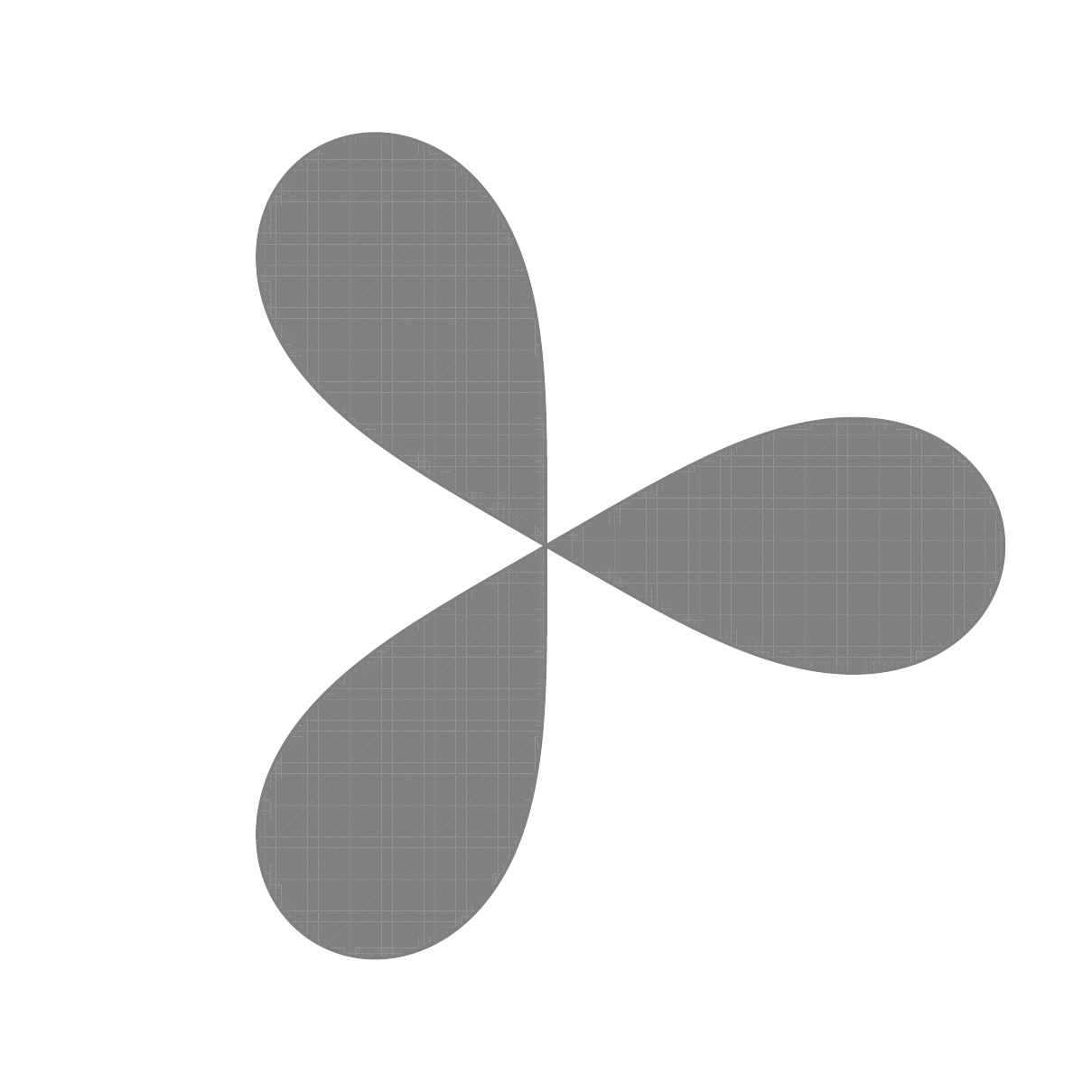}
\hspace{.05\textwidth}
\includegraphics[width=.25\textwidth]{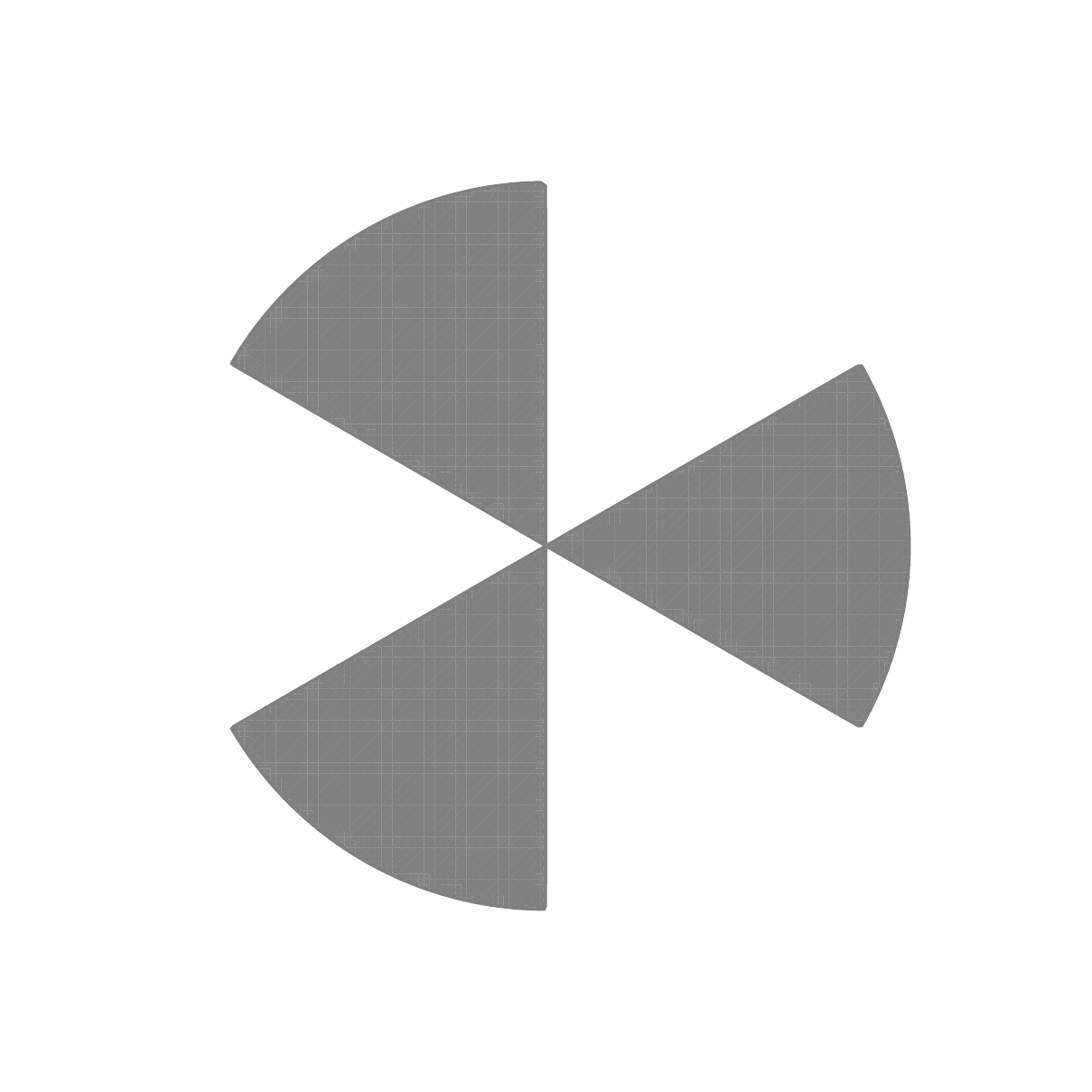}
\end{center}
\caption{The droplet $S$ with $k=3$. In the right hand picture we have zoomed at the origin.}
\label{FIGURE4}
\end{figure}

Let us now consider a limiting holomorphic kernel $L$ in Lemma \ref{th1}.
We write $K(z,w)=L(z,w)e^{-V_0(z)/2-V_0(w)/2}$ and $R(z)=K(z,z)$.

{ We call a subset $\Gamma$ of $\C$ "conical" if $z\in\Gamma$ and $t>0$ imply $tz\in\Gamma$.}

\begin{mth}
If $\Gamma\subset \C^*$ is a closed conical
set such that
$\Gamma\subset\Int \Sigma$, then there is a constant $\alpha=\alpha(\Gamma)>0$ such that
\begin{equation}\label{notr}
R(z)=\Lap Q(z)\cdot (1+O(e^{-\alpha|z|^{2k}})),
\qquad (z\in\Gamma,\quad z\to\infty).
\end{equation}
\end{mth}

\begin{proof}[Remark on the proof]
The proof from the bulk case in \cite{AS2} works also in the present situation. In particular, for
$z\in\Gamma$ with $|z|$ large, there is room inside $S$ to perform H{\"o}rmander estimates near the
corresponding $\zeta=n^{-1/2k}z$. (After all, we just need to be able to squeeze in an
$n^{-1/2}$-neighbourhood about $\zeta$ inside $S$, in order to apply \cite[Lemma 3.3]{AS2}.)
\end{proof}

The lemma shows that there is a unique nontrivial point field with $1$-point function $R$.
Indeed, via the theory in \cite{AKS} we have that $R>0$ on $\C^*$ (on $\C$ if $c\le 0$), and
a Ward equation of the form
$$
\dbar C(z)=R(z)-\Lap V_0(z)-\Lap \log R(z),
$$
holds pointwise on $\C^*$ and in the sense of distributions on $\C$.
It is also easy to see that $R$ enjoys the symmetry
$
R(ze^{2\pi i/k})=R(z).
$

By the general theory in \cite{AKS}, we know that a limiting holomorphic kernel $L$ is the Bergman kernel
of some contractively embedded subspace of the Bergman space $L^2_a(\mu_0)$, which has the reproducing kernel
$$
L_0(z,w)=kE_{\frac 1k,\frac {1+c}k}(z\bar{w}),\qquad E_{a,b}(z)=\sum_{j=0}^\infty \frac {z^j}{\Gamma(aj+b)}.
$$
In the case $k=1$, $c=0$ we know that $L$ has the structure $L(z,w)=F(z+\bar{w})L_0(z,w)$
where $F$ is the free boundary function \eqref{freb} and $L_0(z,w)=e^{z\bar{w}}$.
In other cases, the question of identifying the exact details of a limit
$R(z)=L(z,z)e^{-V_0(z)}$ seems
to be an open problem.

On a related note, the paper \cite{BEG} makes use of theory for
Riemann-Hilbert problems in order to study asymptotic properties of
orthogonal polynomials with respect to the lemniscate ensemble. A somewhat related situation in
a setting of complex geometry, is studied in the paper \cite{ZZ2}.

\subsection{Ordinary singular points on a hard edge}
We will now consider the hard edge case, where we confine the system $\{\zeta_j\}_1^n$ to the
droplet, by redefining $Q$ to be $+\infty$ in the complement $S^c=\C\setminus S$.

An analogue in the Hermitian setting is given by the soft/hard edge ensembles of Claeys and Kuijlaars  \cite{CK}.

We mention without proof the following field-theoretical motivation for studying hard edge ensembles.
In the paper \cite{AHM3} it was shown that, for a free boundary ensemble $\{\zeta_j\}_1^n$, the fluctuations of eigenvalues converge to
a Gaussian free field with {\em free boundary conditions} as $n\to\infty$. If one instead supplies the droplet with hard edge conditions, one obtains {\em Neumann boundary conditions}. Gaussian fields with Neumann boundary conditions have been studied in the recent papers \cite{KT,QW}.

\smallskip

We now describe the setting in detail. Assume that $Q$ is real-analytic in some open set $\Omega\subset\C$.
A compact subset $K$ of $\Omega$ is called a local droplet of $Q$ if
$\Delta Q \cdot \1_K$
is an equilibrium measure of the localized potential
$$Q_K:=Q\cdot \1_K + \infty\cdot \1_{\C\setminus K}.$$

If $S$ is the droplet in potential $Q$, then $S$ is also a local droplet, but we do not obtain all local droplets in this way.
For example, the deltoid in Figure \ref{fig: deltoid} is a local droplet of the
cubic potential
$Q(\zeta)= |\zeta|^2 + \re \zeta^3$, but is not a droplet since $Q$ does not have the required growth near $\infty$, or alternatively, since the deltoid has
$(3,2)$-cusps, see Fig. \ref{fig: deltoid}. This example shows, by the way, that $(3,2)$-cusps, which cannot appear on a free boundary, might well appear on a hard edge.

 \begin{figure}[ht]
 \begin{center}
 \includegraphics[height=.2\textwidth]{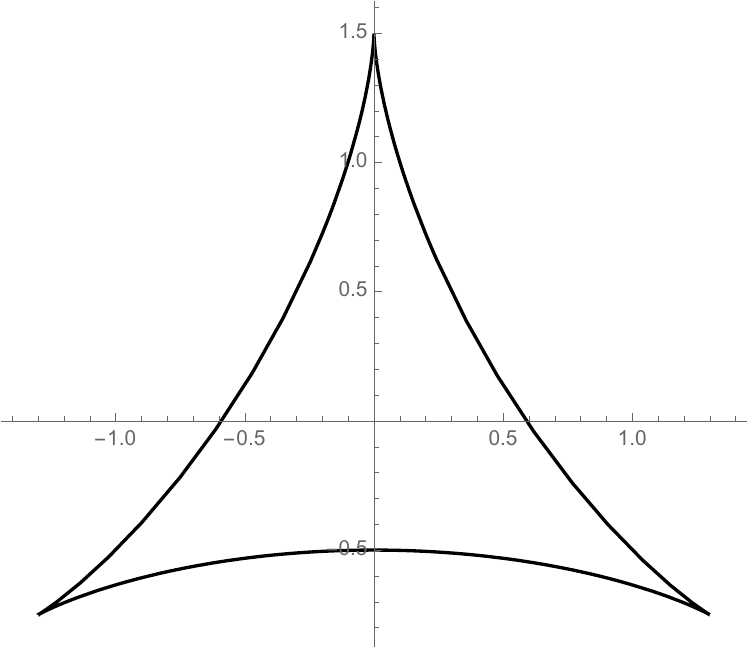}
 \end{center}
 \caption{The deltoid with three maximal $(3,2)$-cusps.}
 \label{fig: deltoid}
 \end{figure}

We now suppose that $Q$ is real-analytic and strictly subharmonic near a point $p\in\d S$ at which $\d S$ has a cusp.

As before, we fix a parameter $\parat>0$ and let $0_n\in S$ be a closest point to $p$ subject to
the condition $\dist(0_n,\d S)=\parat r_n$ where $r_n=1/\sqrt{n\Lap Q(p)}$, see Fig. \ref{fig: moving}.

Taking on the passive interpretation where $0_n=0$ and letting $p$ be on the positive imaginary axis, we rescale by
$$z_j=r_n^{-1}\zeta_j,\qquad j=1,\ldots,n.$$
As before, the limiting rescaled droplet is just the strip
$\Sigma_\parat=\{-\parat\le\re z\le \parat\}.$

The basic structure result in Lemma~\ref{l0} generalizes without difficulty, providing subsequential limiting kernels of the form
\begin{equation}\label{struct2}K(z,w)=G(z,w)\,\Psi(z,w)\, \1_{\Sigma_\parat}(z)\,\1_{\Sigma_\parat}(w).\end{equation}
Here $\Psi$ is an Hermitian-analytic function in the interior of $\Sigma_\parat\times \Sigma_\parat$, which we call the {\em reduced holomorphic kernel}
corresponding to $K$. (The corresponding {\em holomorphic} kernel is $L(z,w)=e^{z\bar{w}}\Psi(z,w)$.)

We remark that a detailed proof of the structure result \eqref{struct2} involves adapting the normal families argument from \cite{AKM} to the present case with a hard edge; details are straightforward, and are therefore skipped.

\smallskip

Applying Ward's identity with potential $Q_S$ and rescaling, precisely as in \cite[Section 4]{AKM}, we find that each limiting $1$-point function $R(z)=K(z,z)$, $R=\lim R_{n_k}$ satisfies
the {\em hard edge Ward's equation} (with parameter $\parat$)
\begin{equation}\label{allvar}\dbar C(z)=R(z)-1-\Lap\log R(z),\qquad z\in\Int\Sigma_\parat\end{equation}
where
\begin{equation}\label{all2}C(z)=\int_{\Sigma_\parat}\frac {B(z,w)}{z-w}\, dA(w),\qquad B(z,w)=\frac {|K(z,w)|^2}{R(z)},\qquad z\in \Int \Sigma_\parat.\end{equation}
Here of course the functions $K$, $B$, $C$
are uniquely determined by the diagonal values $R(z)=\Psi(z,z)=K(z,z)=B(z,z)$, so Ward's equation is a condition for the single function $R$.

 We remark, by contrast to Lemma \ref{l0}, that the inequality $R\le 1$ is false in the hard edge setting (compare Figure \ref{fig4}).

As before, it is natural to assume that the limit $R$ be translation invariant: $R(z)=R(\re z)$. By polarization this means that
$$\Psi(z,w)=\Phi(z+\bar{w})$$
where $\Phi$ is a holomorphic function in $\Sigma_\parat$. We shall assume that $\Phi$ takes the particular form
\begin{equation}\label{erty}\Phi(z)=\gamma*\fii(z)=\int_{-\infty}^{+\infty}\gamma(z-t)\, \fii(t)\, dt\end{equation}
where $\fii$ is a measurable window function of moderate increase (a tempered distribution) on $\R$. Functions of the type \eqref{erty} play an
important role in the sequel; it is convenient to designate them by a special name.

\begin{defn} A function $\Phi$ representable in the form
\eqref{erty} for some window function $\fii$ on $\R$ will be said to be of \textit{error function-type}.
\end{defn}

Note that the functions $F:=\gamma*\1_{(-\infty,0)}$ and $F_\parat:=\gamma*\1_{[-2\parat,2\parat]}$ (from \eqref{freb} and \eqref{ER}) corresponds to the windows $\fii=\1_{(-\infty,0)}$
and $\fii=\1_{[-2T,2T]}$, respectively. We will use both of these kinds of windows in order to construct classes of special functions which model the behaviour
of the particle system near a hard edge.

To this end, we first consider the Hermitian entire function $F_\parat$ defined by the window $\fii=\1_{[-2\parat,2\parat]}$, i.e., \begin{equation} \label{HI}
F_\parat(z)=\gamma*\1_{(-2\parat,2\parat)}(z)=F(z-2\parat)-F(z+2\parat).\end{equation}

Associated to a Borel measurable subset $E\subset\R$
we next define an entire function $H_{E,\parat}$ by
\begin{equation}\label{hedef}H_{E,\parat}(z)=\gamma*\left[\frac {\1_E}{F_\parat}\right](z)=\frac 1 {\sqrt{2\pi}}\int_E\frac {e^{-(z-t)^2/2}}{F_\parat(t)}\, dt,\qquad z\in\C.
\end{equation}

We have the following theorem.

\begin{mth}\label{THV} Suppose that $\Phi$ is of error function-type.
Then the function $R(z)=\Phi(z+\bar{z})\cdot\1_{\Sigma_\parat}(z)$ satisfies Ward's equation \eqref{allvar} in $\Int \Sigma_\parat$ if and only if there is an interval
$I\subset\R$ of positive measure such that $\Phi=H_{I,\parat}$.
\end{mth}

We will also prove a result on limiting reduced kernels $\Psi$ giving rise to solutions to the \textit{mass-one equation} in $\Sigma_\parat$, i.e., to the equation
\begin{equation}\label{moe}\int_{\Sigma_\parat} e^{-|z-w|^2}|\Psi(z,w)|^2\, dA(w)=\Psi(z,z),\qquad (z\in \Sigma_\parat).\end{equation}

\begin{mth}\label{THVI} Suppose that $\Psi$ is translation invariant, $\Psi(z,w)=\Phi(z+\bar{w})$, where $\Phi$ is of error function-type.
Then the mass-one equation \eqref{moe} holds if and only if we have $\Phi=H_{E,\parat}$ where $E$ is some Borel subset of $\R$ of positive measure.
\end{mth}

As a consequence of Theorem \ref{THVI}, we note that the designation
$$
R_{E,T}(z)=H_{E,\parat}(z+\bar{z})\1_{\Sigma_\parat}(z)
$$
associates to each Borel set $E\subset\R$ of positive measure a unique determinantal point field in $\Sigma_\parat$.
Moreover, by Theorem \ref{THV}, the functions $R_{I,T}$ with $I$ an interval
also
give rise to solutions to Ward's equation.

Similar to the free boundary case, we conjecture that each limiting point field is translation invariant and
is determined by the limiting $1$-point function $R=R_{[-2T,2T],T}$.

\begin{figure}[ht]
\begin{center}
\includegraphics[width=.99\textwidth]{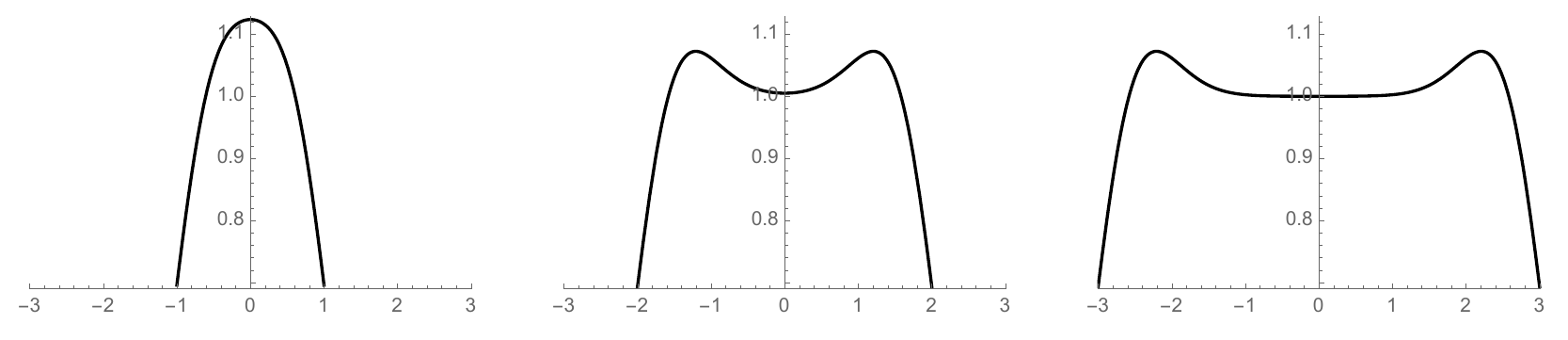}
\end{center}
\caption{The graph of $R_{[-2\parat,2\parat],T}$ restricted to the reals, for $\parat=1$, $\parat=2$, and $\parat=3$.}
\label{fig4}
\end{figure}

One can also equip lemniscate ensembles with a hard edge and prove existence of yet other point fields; details are left to a future investigation.

\subsection{Regular points at a hard edge}

Suppose we rescale about a regular boundary point so that the limiting rescaled droplet is the left half-plane
$\L=\{z;\, \re z\le 0\}$. Recall from \eqref{erty} that a function $\Phi$ is of error function-type if it is
given as a convolution $\gamma*\fii$ of the Gaussian $\gamma$ with some window function $\fii$.

We have the following result.

\begin{mth}\label{TH8}
Suppose that $R(z)=\Phi(z+\bar{z})\cdot \1_\L(z)$ with $\Phi=\gamma*\fii$ of error function-type.
Then $R$ solves Ward's equation
$$\dbar C=R-1-\Lap \log R\quad \text{in}\quad \L,\quad \text{where}\quad C(z):=\int_\L\frac {B(z,w)}{z-w}\, dA(w)$$
if and only if there is an interval $I$
of positive measure such that $$\Phi=H_I:=\gamma*\left[\frac {\1_I} F\right],$$ where as always
$F=\gamma*\1_{(-\infty,0)}$.
Also, $\Phi=\gamma*\fii$ gives rise to a solution to the mass-one equation in $\L$, i.e., the equation
$$\int_{\L} e^{-|z-w|^2}|\Phi(z+\bar{w})|^2\, dA(w)=\Phi(z+\bar{z}),\qquad (z\in \L)$$
if and only if
$\Phi=H_E:=\gamma*\left[\frac {\1_E} F\right]$ for some Borel set $E$ of positive measure.
\end{mth}

\begin{rem}
When we choose $I=(-\infty,0)$ we recover the
hard edge plasma function $H$, which appears in the scaling limit at a regular point on the hard edge corresponding to a radially symmetric potential (cf. \cite{A,AKM}),
$$H(z)=H_I(z)=\int_{-\infty}^0\frac {\gamma(z-t)}{F(t)}\, dt.$$
To our knowledge, it is not known whether this limit holds for non-symmetric potentials, but we believe that this should be the case.

It is interesting to note that the function $H$ appeared already in 1982, in the paper \cite{Sm} due to E.R. Smith, cf. also \cite[Section 15.3.1]{Fo}.
\end{rem}

\subsection{Plan of the paper}
In Section \ref{Ext1} we study the "effective potential" $Q-\eqpot$ locally near a singular boundary point, where $\eqpot$ is the so-called obstacle function.
More precisely, we derive an asymptotic formula for $Q-\eqpot$, which is used in Section \ref{Ext2} to deduce the exterior estimate \eqref{damp} of the 1-point function.

In Section \ref{PMR} we combine our apriori estimates from Section \ref{Ext2}
with the compactness argument in Lemma \ref{TH0}, in order to prove Theorems \ref{THI}-\ref{THIII}.
We also prove Theorem \ref{THIV} on possible translation invariant limits, by using theory for Ward's equation from the paper \cite{AKM}.

The paper is concluded by a complete analysis of translation invariant solutions to Ward's equation with hard edge
confinement in Section \ref{hard_edge}.

\section{Asymptotic expansion of $Q-\eqpot$} \label{Ext1}

\subsection{Plan of this section}
In order to estimate the density $\bfR_n(\zeta)$ near a conformal cusp, we shall first establish an asymptotic expansion of $Q-\eqpot$, where
$\eqpot$ is an auxiliary subharmonic function known as the {\em obstacle function} pertaining to the potential (or "obstacle") $Q$. Using this,
 we shall also prove the important fact that $(3,2)$-cusps do not appear on a free boundary.

\subsection{The obstacle problem near a cusp} \label{tobs}

By definition, $\eqpot$ is the
largest subharmonic function which is bounded above by $Q$ and grows at most as
$\log|\zeta|^2+O(1)$ as $\zeta\to\infty$.
It is well known that this $\eqpot$ is a $C^{1,1}$-smooth function on $\C$ which coincides with $Q$ on $S$
and is harmonic in $\C\setminus S$, with $\eqpot(\zeta)=\log\babs{\zeta}^2+O(1)$ as
$\zeta\to\infty$. (See \cite{HM,ST} for details.)

The reader should note that $Q-\eqpot\ge 0$ everywhere with equality
on $S$ and with strict inequality in $S^c$, except possibly for
some "shallow points"
outside $S$ at which $Q=\check{Q}$, see \cite{HM}.

Now suppose that the droplet has a conformal cusp at the point $0\in\d S$. We assume without loss of generality that $0$ is on the outer boundary, i.e.,
that $0\in \d U$ where $U$ is the unbounded component of $S^c$. We will also assume that the cusp at $0$ points in the negative real direction (as in Fig. \ref{fig: moving}).

Given these proviso, we fix a (surjective) conformal map $\Phi:\HP\to U$ such that $\Phi(0)=0$ and $\Phi(i)=\infty$,
where $\C_+=\{\lambda\in\C;\,\im\lambda>0\}$ is the upper half plane. Note that the outer boundary of $S$ coincides with $\Phi(\R)$ and (since
the cusp is conformal) that $\Phi$ extends analytically to some neighbourhood of the origin.

We can assume that $\Phi'$ has
the Taylor expansion near $\lambda=0$
$$\Phi'(\lambda)=\lambda+a_2\lambda^2+a_3\lambda^3+\cdots.$$
Now form the functions
$$Q_\Phi:=Q\circ\Phi,\qquad\qquad \eqpot_\Phi:=\eqpot\circ\Phi.$$
The function $\eqpot_\Phi$ is harmonic in $\HP$ and extends across $\R$ to a harmonic function $V$.
Write
\begin{equation*}\label{phy}M(\lambda):=\left(Q_\Phi-V\right)(\lambda),\qquad \lambda=\sigma+i\tau.\end{equation*}
Thus $M=(Q-\eqpot)\circ\Phi$ in $\HP$.

Note that $M\ge 0$ and that $M=\d M=0$ on $\R$.

\begin{lem} \label{cusplem} For $\lambda=\sigma+i\tau$, we have
\begin{equation}\label{essex}M(\lambda)= 2\Lap Q(0)\tau^{\,2}\sigma^{\,2}+O(\lambda^{\,5}),\qquad (\lambda\to 0).\end{equation}
\end{lem}

\begin{proof}
We will deduce a full Taylor expansion of $M(\sigma+i \tau)$, which will be useful at later stages.
Notice that
\begin{align}\label{eq:id-deriv-op}
(\partial_\sigma^2+\partial_\tau^2)M(\sigma+i\tau)&=4\Delta M(\sigma+i\tau)\\
\label{latter} \partial_\sigma^jM=\partial_\sigma^j\partial_\tau M&=0\qquad (\text{on}\quad \R),
\end{align}
for $j\ge 1$. To see why \eqref{latter} holds true, just notice that the quadratic vanishing along $\R$ implies that
$M$ may be written as $M(\sigma+i\tau)=\tau^2f(\sigma+i\tau)$ where $f$ is smooth up to $\R$. Hence
$$
\partial_\sigma^j\partial_\tau M(\sigma+i\tau)\big\vert_{\tau=0}=\tau^2\partial_\sigma^j\partial_\tau f(\sigma+i\tau)
+2\tau\partial_\sigma^jf(\sigma+i\tau)\big\vert_{\tau=0}=0.
$$

We proceed to use the identities \eqref{eq:id-deriv-op}-\eqref{latter} to simplify the Taylor expansion of $M$ at the origin.
For any indices $j\ge 0$ and $k\ge 2$ we have
$$
\partial_\sigma^j\partial_\tau^k M(\sigma+i\tau)=\partial_\sigma^j\partial_\tau^{k-2}
(4\Delta Q_\Phi-\partial_\sigma^2M)
=4\partial_\sigma^j\partial_\tau^{k-2}{\Delta} Q_\Phi - \partial_\sigma^{j+2}\partial_{\tau}^{k-2}M.
$$
If we iterate this $l$ times, as long as $k-2l\ge 2$, the right hand side takes the form
$$4\sum_{1\le l\le \lfloor k/2\rfloor}(-1)^{l-1}
\partial_\sigma^{j+2l-2}\partial_\tau^{k-2l}{\Delta}Q_\Phi(\sigma+i\tau)
+{(-1)^{\lfloor k/2\rfloor}}\partial_\sigma^{j+2\lfloor k/2\rfloor}\partial_\tau^{k-2\lfloor k/2\rfloor}M(\sigma+i\tau).
$$
Here, we note that the number $k-2\lfloor k/2\rfloor$ is either $0$ or $1$. Hence, when evaluating
at $\tau=0$ we obtain
$$
\partial_\sigma^j\partial_\tau^k M(\sigma+i\tau)\big\vert_{\tau=0}=
4\sum_{1\le l\le \lfloor k/2\rfloor}(-1)^{l-1}
\partial_\sigma^{j+2l-2}\partial_\tau^{k-2l}{\Delta}Q_\Phi(\sigma+i\tau)\big\vert_{\tau=0}.
$$
For the Taylor expansion of $M$ in $\tau$ this means
\begin{equation}
M(\sigma+i\tau)=\sum_{k\ge 2}\frac{c_k(\sigma)}{k!}\tau^k,
\end{equation}
where
$$
c_k(\sigma)=4\sum_{1\le l\le \lfloor k/2\rfloor}(-1)^{l-1}
\partial_\sigma^{2l-2}\partial_\tau^{k-2l}{\Delta}Q_\Phi(\sigma+i\tau)\big\vert_{\tau=0}.
$$
The first few terms of this expansion read
\begin{multline}\label{lal}M(\sigma+i\tau)=2\Lap Q_\Phi(\sigma)\cdot\tau^{\,2}+\frac 4 {3!}\d_\tau \Lap Q_\Phi(\sigma)\cdot\tau^{\,3}
+\frac 4 {4!}(\d_\tau^{\,2}-\d_\sigma^{\,2})\Lap Q_\Phi(\sigma)\cdot\tau^{\,4}\\
+\frac{4}{5!}(\d_\tau^3-\d_\tau\d_\sigma^2)\Lap Q_\Phi(\sigma)\tau^5+\cdots.\end{multline}
But the Laplacian of $Q_\Phi$ may be computed as follows
$$\Lap Q_\Phi(\sigma+i\tau)=\Lap Q(\Phi(\sigma+i\tau))\babs{\,\Phi'(\sigma+i\tau)\,}^{\,2}=\Lap Q(\Phi(\sigma+i\tau))\cdot(\sigma^{\,2}+\tau^{\,2}+\ldots),$$
so in particular $\Lap Q_\Phi(\sigma)=\Lap Q(\sigma)\cdot (\sigma^2+O(\sigma^3))$.
We have shown that
$$M(\sigma+i\tau)= 2\Lap Q(0)\cdot \sigma^{\,2}\tau^{\,2}+\Ordo(\lambda^{\,5}),\qquad \lambda=\sigma+i\tau\to 0.$$
The proof of Lemma \ref{cusplem} is complete.
\end{proof}

\begin{rem}
We want to thank one of the anonymous referees for suggesting the above short proof.
\end{rem}

\subsection{First application: a preliminary estimate for the 1-point function} Keeping the assumptions in the preceding subsection,
we now rescale about the cusp-point $\zeta=0$ by
\begin{equation}\label{reu}z=i\sqrt{n\Lap Q(0)}\,\zeta.\end{equation}

We shall estimate the rescaled 1-point function $R_n(z)=K_n(z,z)$ and a subsequential limit $R(z)=K(z,z)$.

\begin{lem}\label{2est} For each subsequential limit $R=\lim R_{n_k}$ we have
\begin{equation}\label{aest}R(z)\le Ce^{\,-\,2\, x^{\,2}},\qquad (x=\re z).\end{equation}
\end{lem}

\begin{proof}

Our proof depends on the basic estimate
\begin{equation}\label{dda}\bfR_n(\zeta)\le Cne^{\,-\,n(Q-\eqpot)(\zeta)},\end{equation}
which in fact holds at each
point $p=p_n\in\C$ at which $Q$ is smooth and satisfies a bound $\Lap Q\le C_1$ in some
disc $D(p;c/\sqrt{n})$ with fixed $c>0$.

For completeness, we first outline a proof of the well known estimate \eqref{dda}.

If $f=q\cdot e^{-nQ/2}$ is a weighted polynomial, then the function $F(\zeta)=|f(\zeta)|^2 e^{an|\zeta|^2}$ is logarithmically subharmonic in $D(p;c/\sqrt{n})$, provided that $a>C_1$. It now suffices to apply the sub mean value property of $F$ in that disc, followed by the argument in \cite[Section 3.4]{AKM}. This shows \eqref{dda}.

The estimate \eqref{aest} now follows from the estimate \eqref{dda} and
Lemma~\ref{cusplem}.
Indeed, the estimate \eqref{dda} gives (with a new $C$ depending on $\Lap Q(p)$)
\begin{equation}\label{aestet}R_n(z)\le Ce^{\,-\,nM\left(\lambda_n(z)\right)},\qquad
\text{where}\quad  \lambda_n(z):=\Phi^{\,-1}(-iz/\sqrt{n\Lap Q(p)}).\end{equation}
If $z=x+iy$, then, since $\Phi(\lambda)=\lambda^2/2+O(\lambda^3)$ as $\lambda\to 0$,
\begin{equation}\label{yul}-x=\sqrt{n\Lap Q(p)}\im \left(\lambda^{\,2}/2+\Ordo(\lambda^{\,3})\right)=\sqrt{n\Lap Q(p)}\left(\sigma\tau+\Ordo(\lambda^{\,3})\right),
\quad (\lambda=\sigma+i\tau\to 0).\end{equation}
The estimates \eqref{essex} and \eqref{yul} now give that
$$nM\left(\lambda_n(z)\right)=  2x^{\,2}+\Ordo(n\,\lambda_n(z)^{\,5}),\quad (n\to\infty).$$
Choosing, for example, $|\,z\,|\le \log n$, we see via
\eqref{aestet} that the estimate \eqref{aest} holds.
\end{proof}

\begin{rem} We will improve Lemma \ref{2est} in Section \ref{PMR} by proving that in fact $R\equiv 0$ ("triviality theorem").
Note that formally, Lemma \ref{2est} is just the
special case $T=0$ in Theorem \ref{THII}.
\end{rem}

\subsection{Second application: impossibility of $(3,2)$-cusps}\label{ss:classification}
We shall now show that Lemma \ref{cusplem} excludes the possibility of a $(3,2)$-cusp at the origin (keeping our setting from Subsection \ref{tobs}). As stated earlier,
this generalizes a result due to Sakai \cite{Sa2} concerning the Hele-Shaw case where $\Lap Q=1$ in a neighbourhood of the droplet.

To this end, we first observe (in view of \eqref{wefind}) that the outer boundary admits
a local
parameterization
$$
(\partial S)\cap D(0;\delta)=\{x+i y\in D(0;\delta)\,:\,x=t^2/2,\;\; y=f(t),\;\; t\in [-\epsilon,\epsilon]\},
$$
where
$$
f(t)=c_\nu t^{\nu}+O(t^{\nu+1})
$$
and where $\nu\ge 3$ is the smallest integer such that $c_\nu\ne 0$.

We now obtain two different cases. If
$\nu$ is odd, then the cusp is {\em symmetric} in the sense that the boundary $\d S$ near $0$ is approximated by the union of two symmetric curves $y=\pm c x^{\nu/2}$, $x\ge 0$, where $c\ne 0$ is a constant depending on $c_\nu$.
On the other hand, if $\nu$ is even, then the cusp is {\em bent}, i.e., the droplet is locally given as the region between two graphs of the form $y=c x^{\nu/2}+\cdots$, which have a tangency of the order $\nu/2$ at the cusp.
These two situations are depicted in Fig. \ref{fig: moving}.

In fact, more is true: the only cusps that can appear on the boundary of $S$ satisfy $\nu\not\equiv 3 \mod 4$. We shall here settle
by showing that $(3,2)$-cusps can not appear. (This will be of importance later on.)
A general proof that $(3+4n,2)$-cusps can not appear can be based on \cite[Proposition 4.1]{Sa2}.

\begin{prop} \label{32prop}
A cusp of type $(3,2)$ cannot occur on the boundary $\d S$.
\end{prop}
\begin{proof}
Assume without loss of generality that a cusp of type $(3,2)$ occurs at the origin,
and moreover that it points in the negative real direction. In order to reach a contradiction,
we intend to compute $M(i\tau)$ using \eqref{lal}, and show that
$M(i\tau)$ must take on negative values arbitrarily close to $0$.
Since the cusp is assumed to be of type $(3,2)$ the conformal mapping $\Phi$
takes the form
$$
\Phi(\lambda)=\frac12 \lambda^2+\frac{a+i b}{3}\lambda^3+O(\lvert \lambda\rvert^4)
$$
where $b\ne 0$, from which it follows that
\begin{align*}
\Lap Q_\Phi(\lambda)&=\Lap Q(\Phi(\lambda))\babs{\Phi'(\lambda)}^2\\
&=\Lap Q(\Phi(\lambda))\left[\sigma^2+\tau^2+2a(\sigma^3+\sigma \tau^2)-2b(\tau^3+\sigma^2\tau)+
O(\lvert \lambda\rvert^4)\right].
\end{align*}
when $\lambda=\sigma+i\tau\to0$.
A computation of the first five coefficients in the expansion \eqref{lal} now shows that, as $\tau\to 0$,
\begin{align*}
M(i \tau)&=\frac 4 {5!}(\d_\tau^3-\d_\tau\d_\sigma^2)\Lap Q_\Phi(\sigma)\tau^5+O(\tau^6)\\
&=\frac 4 {5!}(-2b)(3!-2!)\Lap Q(0)\tau^5+O(\tau^6)=-\frac 4 {15}b\Lap Q(0)\tau^5+\cdots
\end{align*}
from which the assertion follows.
\end{proof}

 \begin{ex} $(5,2)$-cusps actually do appear on the free boundary of some droplets. To see this, one can consider potentials of the form
 $Q_t(\zeta)=(1/t)|\zeta|^2-2c\log|\zeta-a|-2c\log|\zeta-\bar{a}|$ where $c>0$ and $a$ is a non-real complex number. Here the parameter $t$ equals to the area
 of the droplet (divided by $\pi$). Fixing $a$ and $c$ suitably, the droplet develops a $(5,2)$-cusp for a certain critical value $t=t_0$, see
 Figure~\ref{fig7}. We are grateful to S.-Y. Lee and M. Yang for communicating this example.
\begin{figure}[ht]
\begin{center}
\includegraphics[width=.25\textwidth]{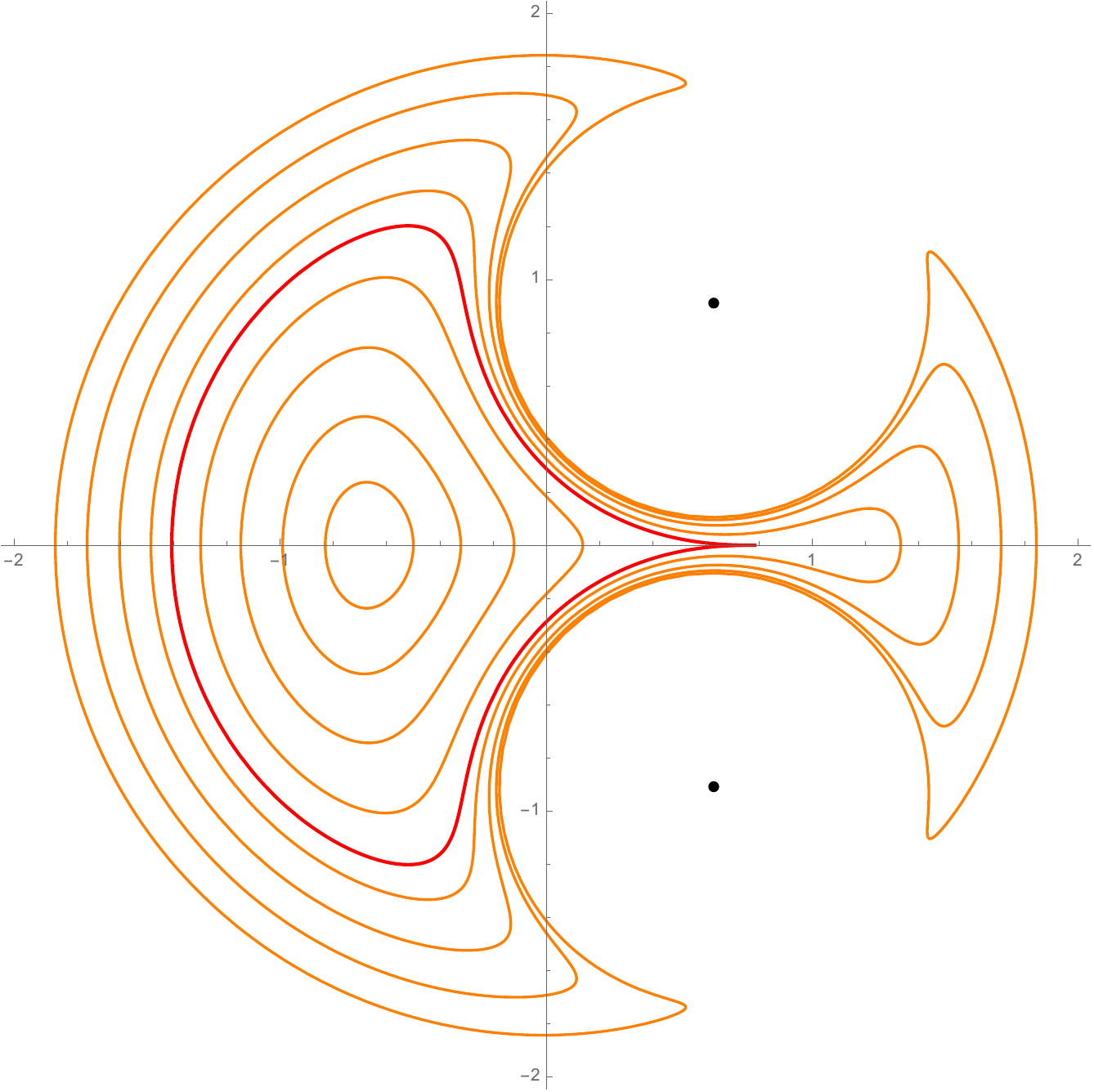}
\end{center}
\caption{Droplets pertaining to $Q_t$; one of them has a $(5,2)$-cusp.}
\label{fig7}
\end{figure}
\end{ex}

\section{Exterior estimates near singular points: Proof of Theorem \ref{THII}} \label{Ext2}
The goal of this section is to prove the estimate $R(z)\le Ce^{-2(|x|-T)^2}$ in Theorem \ref{THII}, where $R$ is the suitably rescaled $1$-point function
near a singular point and $x=\re z$. The important point to bear in mind is that, if we restrict to $z$ with $|z|\le M$ for some large $M$,
then the rescaled droplet is  a good approximation of the strip $-T\le \re z\le T$.

\smallskip

We start with the case, where the boundary $\d S$ has an ordinary $(\nu,2)$-cusp
at the point $0$, pointing in the negative real direction. Here $\nu\ge 4$ is an
integer. The case of a double point is rather more trivial, and will be handled
afterwards.

Fix $\parat>0$ and a large integer $n$ and write $\delta_n=\parat/\sqrt{n\Lap Q(0)}$. In the following, we consider the non-rescaled droplet to sit in the $\zeta=\xi+i\eta$ plane.

Let $p_n\in \Int S$ be the unique point closest to $0$ such that
$D(p_n;\delta_n)\subset S$; this means that the boundary circle $\{|\zeta-p_n|=\delta_n\}$ is tangent to $\d S$ at two points (see Figure \ref{fig: moving} or Figure \ref{qnc}).

Let $q_n$ denote one of the two points in $\{|\zeta-p_n|=\delta_n\}\cap(\d S)$, say, the upper one, as in Fig. \ref{qnc}.

Notice that both $p_n$ and $q_n$ converge to $0$ as $n\to\infty$, and that we may for example replace $\Lap Q(q_n)$ by $\Lap Q(0)$ with a vanishing error in the limit, as $n\to\infty$.

We shall start by deducing the asymptotic relation
\begin{equation}\label{eq:est-pn-qn}
\babs{\,p_n\,}\gtrsim n^{-1/\nu},\qquad n\to\infty
\end{equation}

\begin{figure}[ht]
\begin{center}
\includegraphics[width=.3\textwidth]{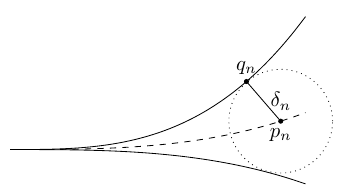}
\end{center}
\caption{$(\nu,2)$-cusp with odd $\nu$.}
\label{qnc}
\end{figure}

To prove the estimate \eqref{eq:est-pn-qn}, we consider first the case when $\nu$ is odd. In this case the cusp is "symmetric" as in Figure \ref{qnc}, i.e., there exists a number $c>0$ such that $\d S$ takes the form
\begin{equation}\label{straightcusp}\eta=\pm c \xi^{\nu/2}(1+o(1)),\qquad (\zeta=\xi+i\eta).\end{equation}
(Cf. \eqref{wefind}.)
This implies that the point
$p_n$ is approximately real and positive for large $n$.

As the distance from a point $\xi\in\R_+$ to the curve \eqref{straightcusp}
is no larger than the vertical distance, which is $c\xi^{\nu/2}(1+o(1))$, it follows that
$|p_n|^{\nu/2}\gtrsim \delta_n=O(n^{-1/2})$. 

Next the case when $\nu$ is even, i.e., when the cusp is bent as in the right hand picture in Figure \ref{fig: moving}.
Let the graph $\eta=f(\xi)$, $\xi\ge 0$ parameterize the arc of the cusp which lies farthest from the $\xi$-axis (the upper arc in
Fig. \ref{fig: moving}).

We roughly
estimate $\babs{\,p_n\,}$ as follows. It is evident that $\babs{\,p_n\,}\ge \xi_n$, where $\xi_n\in\R_+$ is the unique point closest to the origin, which lies at a distance $\delta_n$
from
union of the two curves
$
\{\eta=\pm f(\xi);\, \xi\ge 0\}.
$
These curves define a symmetric cusp, to which our earlier argument applies. As a result, we find that $\xi_n\gtrsim n^{-1/\nu}$. From this it follows directly that $\babs{\,p_n\,}\gtrsim n^{-1/\nu}$,
whence
\eqref{eq:est-pn-qn} is shown also for bent cusps.

A few remarks are in order. First, if $\nu$ is odd, an examination of the above argument shows that we have the stronger asymptotic $|p_n|\asymp n^{-1/\nu}$.
If $\nu$ is even, $|p_n|$ might be larger, but there is always an $\epsilon>0$ (depending on the cusp) such that $n^\epsilon|p_n|\to 0$ as $n\to\infty$.
Finally, as $|q_n|=|p_n|+O(n^{-1/2})$, corresponding estimates hold also for $|q_n|$.

\smallskip

After these preliminaries, we rescale about $q_n$ as follows. Let $e^{i\theta_n}$ be the outer normal to $\d S$ at $q_n$ and put
\begin{equation}\label{re0}z=z(\zeta)=e^{-i\theta_n}\sqrt{n\Lap Q(0)}\,(\zeta-q_n).\end{equation}
Then the rescaled droplet (restricted to a compact subset of the $z$-plane) looks roughly like the strip
$$-2T<\re z<0,$$
and $z(q_n)=0$.

Let $\Phi$ be a conformal map $\C_+\to U$ where $\C_+$ is the upper half-plane
and $U$ is the component of $S^c$ containing $\infty$. We assume that
$\Phi(0)=0$ and
\begin{equation}\label{via}\zeta=\Phi(\lambda)=\frac 1 2 \lambda^2+O(\lambda^3),\qquad \lambda\to 0.\end{equation}
As before, $\Phi(\R)$ parameterizes the outer boundary $\d U$ of $S$, and since the cusp at $0$ is conformal, $\Phi$ extends analytically across $\R$.

Now let $\sigma_n$ be the point in $\R$ such that $\Phi(\sigma_n)=q_n$. We assume without loss that $\sigma_n>0$.

Locally near
the point $q_n$ there is an inverse mapping to $\Phi$ of the form
\begin{equation}\label{duct}\lambda=\Phi^{-1}(\zeta)=(2\zeta)^{1/2}+O(\zeta).\end{equation}

Now fix a point $z$ with $|z|\le \log n$ and put
\begin{equation}\label{nyen}z_n=\frac {ze^{i\theta_n}}{\sqrt{n\Lap Q(0)}},\qquad \zeta=q_n+z_n.\end{equation}
Then the relation $z=z(\zeta)$ in \eqref{re0} holds and $|z_n|\lesssim(\log n)/\sqrt{n}$.

At this point, we note the following gradient bound for $\Phi^{-1}$,
\begin{equation}\label{gradbound}\sup_{\zeta\in D_n}|(\Phi^{-1})'(\zeta)|\lesssim n^{1/2\nu},\qquad D_n:=D\left(q_n;\frac {\log n}{\sqrt{n\Lap Q(0)}}\right).
\end{equation}
To prove this it suffices to note that for $\zeta\in D_n$ we have
$(\Phi^{-1})'(\zeta)=(2\zeta)^{-1/2}+O(1)$ and
$|\zeta|\gtrsim n^{-1/\nu}$ by \eqref{duct} and \eqref{eq:est-pn-qn}, respectively.

Next we define the complex number $\eps_n=\alpha_n + i\beta_n$ by $\sigma_n+\eps_n=\Phi^{-1}(\zeta)=\Phi^{-1}(q_n+z_n)$.
Then by \eqref{gradbound} we obtain immediately
\begin{equation}\label{til}|\,\varepsilon_n\,|=|\Phi^{-1}(q_n+z_n)-\Phi^{-1}(q_n)|\lesssim n^{1/2\nu}|z_n|\lesssim \frac {\log n} {n^{1/2-1/(2\nu)}},\end{equation}
and hence (since $\nu\ge 4$)
\begin{equation}\label{rak}n\,|\,\varepsilon_n\,|^3\lesssim
n\frac {\log^3 n}{n^{3/2-3/(2\nu)}}=\frac {\log^3 n}{n^{(\nu-3)/(2\nu)}}
\to 0,\quad (n\to\infty).\end{equation}

By \eqref{via} and \eqref{eq:est-pn-qn} we have the estimates
$|\Phi'(\sigma_n)|\asymp \sigma_n\asymp|q_n|^{1/2}\gtrsim n^{-1/2\nu}$ and by \eqref{til} we see that $|\eps_n|=o(\sigma_n)$ so
Taylor's formula gives that, as $n\to\infty$,
$$|\sigma_n\eps_n|\asymp |\sigma_n\eps_n+O(\eps_n^2)|\asymp |\Phi(\sigma_n+\eps_n)-q_n|=|z_n|\lesssim (\log n)/\sqrt{n}.$$
From this we draw the conclusion that
\begin{equation}\label{beno}n|\Phi'(\sigma_n)|^2\beta_n^2\lesssim n|\sigma_n \eps_n|^2\lesssim \log^2 n.\end{equation}

After these observations, we now prove the required decay about the moving point $q_n$.

\begin{lem} \label{l02} Let $R_n(z)$ be the rescaled $1$-point function according to
the rescaling \eqref{re0}. There is then a constant $C$ such that
$R_n(z)\le Ce^{-2x^2}$ when $\babs{z}\le \log n$ and $x=\re z>0$.
\end{lem}

\begin{proof}
We know that $\bfR_n(\zeta)\le Cne^{-n(Q-\check{Q})(\zeta)}$. Since $\zeta=\Phi(\sigma_n+\eps_n)$,
it will suffice to show that
\begin{equation}\label{check}2\,x^2=n(Q_\Phi-\check{Q}_\Phi)(\sigma_n+\varepsilon_n)+o(1),\end{equation}
where $z=x+iy$ is related to $\zeta$ via \eqref{re0}.

However, by the estimate \eqref{lal} we have
$$n(Q_\Phi-\check{Q}_\Phi)(\sigma_n+\varepsilon_n) =
2n\Delta Q(\Phi(\sigma_n + \alpha_n))|\Phi'(\sigma_n+\alpha_n)|^2 \beta_n^2 + O(n\beta_n^3).$$
It hence follows from \eqref{til}~--~\eqref{beno} that
\begin{equation}\label{zock}n(Q_\Phi-\check{Q}_\Phi)(\sigma_n+\varepsilon_n)=2n\Delta Q(0) \babs{\Phi'(\sigma_n)}^2 \beta_n^2+o(1).\end{equation}
Indeed (using also the fact that $\sigma_n{\cdot \log^2 n}\to0$ as $n\to\infty$) we have
\begin{align*}n\Delta Q(\Phi(\sigma_n + \alpha_n))|\Phi'(\sigma_n+\alpha_n)|^2 \beta_n^2 &- n\Delta Q(0)\babs{\Phi'(\sigma_n+\alpha_n)}^2 \beta_n^2\\
& =  O(n|\sigma_n + \alpha_n|^3\beta_n^2) = o(1).
\end{align*}
Similarly, it is straightforward to check that
\begin{align*}n\Delta Q(0)\babs{\Phi'(\sigma_n+\alpha_n)}^2 \beta_n^2& - n\Delta Q(0)\babs{\Phi'(\sigma_n)}^2 \beta_n^2 \\
&= O(n|\sigma_n\alpha_n|\beta_n^2) = o(1).
\end{align*}
Inserting in \eqref{nyen} the Taylor expansion of $\Phi$ about $\sigma_n$ we find successively that
\begin{align*}z&=e^{-i\theta_n}\sqrt{n\Delta Q(0)}\left(\Phi'(\sigma_n)(\alpha_n+i\beta_n)+O(|\varepsilon_n|^2)\right), \qquad e^{i\theta_n} = i \frac{\Phi'(\sigma_n)}{|\Phi'(\sigma_n)|},\\
x&=|\Phi'(\sigma_n)|\cdot \sqrt{n\Lap Q(0)}\cdot \beta_n+O(\sqrt{n}|\eps_n|^2).\end{align*}
Since $n|\eps_n|^3\to 0$ (see \eqref{rak}) we infer that
$$2x^2=2n\Delta Q(0) \babs{\Phi'(\sigma_n)}^2 \beta_n^2+o(1).$$
This finishes the proof of the lemma, in view of \eqref{zock}.
\end{proof}

We turn to the estimate \eqref{damp} in Theorem \ref{THII}.

Recall the point $p_n$ defined at the beginning of the preceding subsection, i.e., the
closest point $p_n\in\Int S$ to $0$, having distance $T/\sqrt{n\Lap Q(0)}$ to the boundary.

We now slightly modify the rescaling so that the point $p_n$ is mapped to the origin instead of $q_n$,
$$z=e^{-i\theta_n}\sqrt{n\Lap Q(0)}(\zeta-p_n),$$
and we write $R_n$ for the corresponding rescaled $1$-point function. Since
$$\babs{\,p_n-q_n\,}=
T/\sqrt{n\Lap Q(0)}$$
we obtain from Lemma \ref{l02} the estimate
\begin{equation}\label{moby}R_n(z)\le Ce^{-\,2\,\left(\,\babs{\,x\,}\,-\,T\,\right)^{\,2}},\qquad |z|\le\log n,\quad x=\re z.\end{equation}
This finishes the proof of Theorem \ref{THII}, eq. \eqref{moby}, in the case of a cusp.

There remains only to treat the case of a double point. This follows as in the case of a regular boundary point, using the
estimate in \cite[Lemma 5.5]{AKM}, which works in both directions normal to the boundary near the
double point. After all, the double point is just another interior point of each of the analytic arcs which meet at the double point, see Fig. \ref{fig1}.
Hence the estimate $R_n(z)\le Ce^{-2(|x|-T)^2}$ follows easily for the 1-point function $R_n$ rescaled about
one of the points $0_n'$ or $0_n''$ appearing in Theorem \ref{THII}.

By this, the estimate $R(z)\le Ce^{-2(|x|-T)^2}$ in Theorem \ref{THII} is completely proved. $\qed$

\section{Free boundary ensembles} \label{PMR}
We shall now prove Theorems \ref{THI}-\ref{THIV}.

\subsection{The triviality theorem}
We now prove Theorem \ref{THI}. Suppose that $p$ is either a double point or a cusp of type $(\nu,2)$ where
$\nu>3$ and that $\Lap Q(p)>0$ and rescale about $p$ according to
$$z=e^{-i\theta}\sqrt{n\Lap Q(p)}(\zeta-p)$$
where $e^{i\theta}$ is one of the normal directions to $\d S$ at $p$.

Let $K=G\Psi$ be a limiting kernel. We must prove that the limiting $1$-point function $R(z)=K(z,z)=\Psi(z,z)$ vanishes identically. To this end, we shall use the corresponding holomorphic kernel
$$L(z,w)=e^{z\bar{w}}\Psi(z,w).$$

We now call on the result in \cite[Lemma 4.3]{AKM}, which says that
the function $S(z):=|z|^2+\log R(z)$ is subharmonic. Combining this with the estimate $R(z)\le Ce^{-2x^2}$ for some constant $C$, obtained in
Lemma \ref{2est} for cusps (the same estimate holds at a double point, with a much easier proof), we
deduce the bound
$$S(z)\le \log C+y^2-x^2.$$
But $y^2-x^2$ is harmonic, so the function $\tilde{S}=S-(y^2-x^2)$ is subharmonic and bounded above by $\log C$. Hence it is constant, i.e.,
$$R(z)=Ce^{-\,2\,x^{\,2}}$$
for a (new) constant $C$. If $R$ is nontrivial we can assume that $C=1$. By polarization, then
$$\Psi(z,w)=e^{-(z+\bar{w})^{\,2}/2},$$
so the kernel $L(z,w)=e^{z\bar{w}}\Psi(z,w)$ must satisfy
\begin{align*}\int\babs{\,L(0,w)\,}^{\,2}e^{-\,|\,w\,|^{\,2}}\, dA(w)=\int\babs{\,\Psi(0,w)\,}^{\,2}e^{-\,|\,w\,|^{\,2}}\, dA(w)=\int e^{-\,2\,x^{\,2}}\, dA=\infty.
\end{align*}
This contradicts the mass-one inequality \eqref{moi}, so we must have $C=0$. $\qed$

\smallskip

We are grateful to H. Hedenmalm for communication in connection with the above proof, \cite{Hed}.

\subsection{Proof of the existence theorems}
We now prove Theorem \ref{THII} and Theorem \ref{THIII}.

Let $p$ be either a $(\nu,2)$-cusp with $\nu\ge 4$ or a double point.
In both cases we assume that $\Lap Q(p)>0$. Also fix a number $T> 0$. For a given $n\in\Z_+$, we let $p_n$
be a point in $S$ whose distance to the boundary is $T/\sqrt{n\Lap Q(p)}$ and whose distance
to $p$ is minimal.

We rescale about
$p_n$,
$$z_j=e^{-i\theta_n}\sqrt{n\Lap Q(p)}(\zeta_j-p_n),\qquad j=1,\ldots,n,$$
where the angle $\theta_n$ is chosen so that the image of the cusp point $p$ lies
on the positive imaginary axis.

Note that as $n\to\infty$, the
image of $S$ near $p_n$ looks approximately like the strip
\begin{equation}\label{strip}\Sigma_T:\qquad  -T <\re z<T.\end{equation}

Let $K_n$ be a correlation kernel of the rescaled system $\config_n=\{z_j\}_1^n$. We write
$R_n(z)=K_n(z,z)$. By Lemma \ref{l0} we know that
there is a sequence of cocycles $c_n$ such that every subsequence of $c_nK_n$ has a subsequence converging to $G\Psi$ where $\Psi$ is some
Hermitian entire function. It remains only to show that the function $R(z)=\Psi(z,z)$ does not vanish identically if $T$ is large enough.

To this end, we shall use the following estimate found in \cite[Theorem 5.4]{AKM},
$$\babs{\,\bfR_n(\zeta)-n\Lap Q(\zeta)\,}\le C\left(1+ne^{-n\ell\Lap Q(\zeta)\cdot \delta(\zeta)^2}\right),\quad \zeta\in S,$$
where $\ell$ is a positive constant
and $\delta(\zeta)=\dist(\zeta,\d S)$. If we choose $\zeta=p_n$ where $\delta(p_n)=T/\sqrt{n\Lap Q(p_n)}$, we obtain
for the rescaled $1$-point function $R_n$ that
$$\babs{\, R_n(0)-1\,}\le C(1/n+e^{-\ell T^2}).$$
Choosing $n$ and $T$ sufficiently large that the right hand side is strictly less than $1$, we obtain that $R(0)>0$. By Lemma \ref{l0} we then have $R>0$ everywhere on $\C$.
$\qed$

\subsection{Translation invariant candidates} We finally prove Theorem \ref{THIV}. Suppose the $1$-point function $R(z)=\Phi(z+\bar{z})$ is translation invariant.
If $R$ is nontrivial, then $R$ gives rise to a solution to Ward's equation by Lemma \ref{l0}. Hence we can use \cite[Theorem 1.6]{AKM} to conclude that $\Phi$ has the structure
$$\Phi(z)=\gamma*\1_I(z)=\frac 1 {\sqrt{2\pi}}\int_I e^{-(z-t)^2/2}\, dt,$$
where $I\subset \R$ is an interval of positive measure. By the estimate $R(z)\le Ce^{-2(|z|-T)^2}$ (Theorem \ref{THII}), we see
that $I$ must be included in the interval $[-2T,2T]$. $\qed$

\section{Hard edge point fields in a strip} \label{hard_edge}

In this section, we prove Theorems \ref{THV}, \ref{THVI}, and \ref{TH8}. For this, we fix a parameter $\parat>0$ and let $\Sigma_\parat$
denote the symmetric strip of width $2\parat$
$$
\Sigma_\parat =\{z = x+iy\,|\, x\in [-\parat,\parat]\}.
$$

\subsection{Some preliminaries}

Given a Hermitian entire function $\Psi$, we put
\begin{align}R(z)&=\Psi(z,z)\cdot\1_{\Sigma_\parat}(z),\\
D(z)&=\int_{\Sigma_\parat}\frac {e^{-|z-w|^2}}{z-w}\babs{\,\Psi(z,w)\,}^{\,2}\, dA(w).
\end{align}
Note that
$D=RC$ on $\Sigma_\parat$, where $C(z)$ is the Cauchy transform defined in \eqref{all2}.

\begin{lem} \label{Ward:4} Ward's equation  is satisfied on $\Int \Sigma_\parat$ if and only if there is a smooth function
$P$ on $\Int {\Sigma_\parat}$ such that
\begin{equation}\label{Ward:5}\dbar P=R-1\quad \text{and}\quad D=PR-\d R\quad \text{on}\quad \Int {\Sigma_\parat}.\end{equation}
\end{lem}

\begin{proof} Ward's equation means that
\begin{equation}\label{Ward:6}\dbar(D/R)=R-1-\dbar(\d R/R)\quad \text{on}\quad {\Sigma_\parat}.\end{equation} If we let
$P_0$ be an arbitrary solution to $\dbar P_0=R-1$, then this can be written
\begin{equation*}\dbar\left(\frac D R-P_0+\frac {\d R} R\right)=0\quad \text{on}\quad{\Sigma_\parat}.\end{equation*}
The last identity means that there is a holomorphic function $E$ on ${\Sigma_\parat}$ such that
$D-P_0R+\d R=ER$. Letting $P=P_0+E$ we now see that the conditions in \eqref{Ward:5} are satisfied.
Conversely, if the conditions in \eqref{Ward:5} hold, then \eqref{Ward:6} holds since
$\dbar(D/R)=\dbar(P-\d R/R)=R-1-\dbar(\d R/R)$.
\end{proof}

Next we assume translation invariance $\Psi(z,w)=\Phi(z+\bar{w})$ and introduce the function
$$L(x):=D(x/2)$$ for $x\in (-2\parat,2\parat)$, so
\begin{equation}\label{ldef}L(x)=-\int_{\Sigma_\parat-x/2}\frac {e^{-|w|^2}}w\babs{\Phi(x+w)}^2\, dA(w).
\end{equation}

\begin{lem}\label{303} An error type-function $\Phi=\gamma*\fii$ gives rise to a solution to Ward's equation
if and only if there is a smooth function $G(x)$ ($x\in I$) of the form
\begin{equation}\label{gform}G=\gamma*g,\qquad \fii=g'+1,\end{equation}
such that
\begin{equation}\label{ttw}L=G\Phi-\Phi'\quad \text{and}\quad G'=\Phi-1\quad \text{on}\quad I.
\end{equation}
\end{lem}

\begin{proof} It is easily seen that the equation \eqref{ttw} is the same as \eqref{Ward:5}, where $G(x)=P(x/2)$.
In particular, we have $G'=\Phi-1=\gamma*(\fii-1)$. Taking primitive functions, it follows that $G=\gamma*g$ where $g$ is some function
with $g'=\fii-1$. This proves \eqref{gform}.
\end{proof}

\subsection{The Gaussian semi-group} We will use the Fourier transform with normalization
$$\calF\left[f\right](t)=\hat{f}(t)=\int_{-\infty}^{+\infty}f(x)e^{-itx}\, dx.$$
Hence $\calF\left[f*g\right]=\hat{f}\hat{g}$ where "$*$'' is the usual convolution product in $\R$.

Let $\chi_a(x)=e^{-ax^2/2}=(\sqrt{2\pi}\gamma(x))^a$, where $a>0$, $\gamma(x)=\frac 1 {\sqrt{2\pi}}e^{-x^2/2}$.
We have that $\hat{\gamma}=\sqrt{2\pi}\gamma$ and more generally $\hat{\chi}_a=\sqrt{\frac {2\pi} a}\chi_{1/a}$. Hence
$$\chi_{1/a}*\chi_{1/b}=c\chi_{1/(a+b)},$$ where $c=\sqrt{\frac {2\pi ab}{a+b}}$.

\subsection{Generalized Fourier transform and analytic continuation} \label{Fourier:sec} If $g$ is a suitable test-function on $\R$
(e.g. $g\in L^\infty(\R)$), the convolution
$$G(z):=\gamma*g(z)=\int_{-\infty}^{+\infty}\gamma(z-t)g(t)\, dt$$
defines an entire function, which is the analytic continuation of $\gamma*g(x)$ to $\C$. For a function $G$
of this form, we define the Fourier transform by
$$\hat{G}(t):=\hat{\gamma}(t)\hat{g}(t)=\sqrt{2\pi} \gamma(t)\hat{g}(t).$$

By Fourier's inversion formula, we then have
$$G(z)=\frac 1 {2\pi}\int_{-\infty}^{+\infty}\hat{G}(t)e^{izt}\, dt.$$
This can be seen as another method of analytic continuation.

It follows that for suitable test-functions (or tempered distributions) $g$ we have
\begin{equation}\label{an:cont}\int_\R\gamma(z-t)g(t)\, dt=\frac 1 {\sqrt{2\pi}}\int_\R \gamma(t)\hat{g}(t)e^{izt}\, dt.
\end{equation}

\subsection{Faddeeva's formula} Consider the complementary error function:
$$\erfc(z)=\frac 2 {\sqrt{\pi}}\int_z^{+\infty}e^{-t^2}\, dt.$$
The following theorem is well-known in the plasma literature (e.g. \cite{FC}).

\begin{thm} (Faddeeva's formula)
$$\frac i \pi\int_{-\infty}^{+\infty}\frac {e^{-t^2}}{z-t}\, dt=\begin{cases}
e^{-z^2}\erfc(-iz),& \im z>0,\cr
-e^{-z^2}\erfc(iz),& \im z<0.\cr
\end{cases}$$
\end{thm}

The theorem follows easily by integration by parts and the observation that
$$\frac 1 {z-t}=(-2i)\int_0^\infty e^{2i(z-t)u}\, du,\qquad (\im z >0).$$

It is instructive to give an alternative argument, based on the formula \eqref{an:cont}.

\begin{proof}[Proof of Faddeeva's formula] Let $g=\1_{(-\infty,0)}$. The Fourier transform is $\hat{g}(t)=\frac i {t-0i}$.
Inserting this into \eqref{an:cont}, we find that
$$\int_{-\infty}^0\gamma(iz-t)\, dt=\frac i{\sqrt{2\pi}}\int_\R \frac {\gamma(t)} {t-0i} e^{-zt}\, dt.$$
Here the left hand side is $F(iz)=\frac 1 2 \erfc\left(iz/\sqrt{2}\right)$ while the right hand side is
$$\frac 1 {\sqrt{2\pi}}\int_\R \frac {\gamma(t)} {t-0i}e^{-zt}\, dt=\frac i {2\pi}e^{z^2/2}\int_\R \frac {e^{-(t+z)^2/2}}{t-0i}\, dt=\frac i {2\pi}
e^{z^2/2}\int_\R \frac {e^{-t^2/2}}{t-z}\, dt,\quad(\im z<0),$$
where the last equality can be justified using Cauchy's theorem.
We have shown that
$$\frac i {2\pi}\int_\R \frac {e^{-t^2/2}}{z-t}\, dt=-\frac 1 2e^{-z^2/2}\erfc\left(iz/\sqrt{2}\right)\quad \text{when}\quad \im z<0.$$ This
is equivalent to Faddeeva's formula.
\end{proof}

\subsection{Auxiliary identities} \label{plasma:functions} We will use two {\em plasma-functions}:
\begin{align*}
F(z)&=\gamma*\1_{(-\infty,0)}(z)=\frac 1 2 \erfc\left(\frac z {\sqrt{2}}\right),\\
E(z)&=\gamma*F(z)=\frac 1{2}\erfc\left(\frac z 2\right).
\end{align*}

Recall that by Lemma \ref{303}, a holomorphic function $\Phi$ gives rise to a solution to Ward's equation
if and only if there is a smooth function $G(x)$ $(-2\parat<x<2\parat)$ such that
\begin{equation}\label{Ward:7}L=G\Phi-\Phi'\quad \text{and}\quad G'=\Phi-1\quad \text{on}\quad (-2\parat,2\parat).
\end{equation}

We shall need to compute the "transforms"
\begin{equation}K(x,s,t):=\int_{\L-x/2}\frac {e^{-|w|^2}}{w}e^{iwt}e^{i\bar{w}s}\, dA(w),\qquad(x,s,t\in\R)
\end{equation}
and
\begin{equation}K_\parat(x,s,t):=\int_{{\Sigma_\parat}-x/2}\frac {e^{-|w|^2}}{w}e^{iwt}e^{i\bar{w}s}\, dA(w),\qquad(x\in(-2\parat,2\parat),s,t\in\R).
\end{equation}
For $x\in(-2\parat,2\parat),$ we have
$$K_\parat(x,s,t) = K(x-2\parat,s,t) - K(x+2\parat,s,t).$$
\begin{lem} \label{Fourier:2}
We have that
\begin{equation*}
iK(x,s,t)=
\begin{cases}
\dfrac {e^{-st}}{s} E(x+i(s+t))
-\dfrac {e^{-isx}}{s}E(x+i(t-s)), &(x\ge 0),\\
&\\
\dfrac {e^{-st}}{s} E(x+i(s+t))
-\dfrac {e^{-isx}}{s}E(x+i(t-s))+ \dfrac {e^{-isx}-1}s, \,\, &(x\le 0).
\end{cases}
\end{equation*}
In particular,
\begin{equation}\label{FT:4}iK(x,0,t)=
\begin{cases}
i(x+it)E(x+it)+2iE'(x+it),\quad &(x\ge 0),\\
i(x+it)E(x+it)+2iE'(x+it)-ix,\quad &(x\le 0).
\end{cases}
\end{equation}
\end{lem}

\begin{proof}
Note that, with $w=u+iv$,
\begin{align*}K&=\frac 1 \pi\int_{-\infty}^{-x/2}e^{-u^2+iu(t+s)}\, du
\,\int_{-\infty}^{+\infty}\frac {e^{-v^2-v(t-s)}}
{u+iv}\, dv\\
&=\frac 1 \pi e^{(t-s)^2/4}\int_{-\infty}^{-x/2}e^{-u^2+iu(t+s)}\, du
\,\int_{-\infty}^{+\infty}\frac {e^{-(v+(t-s)/2)^2}}
{u+iv}\, dv.
\end{align*}
Writing $\xi=v+(t-s)/2$,
the inner integral becomes (say, if $u< 0$)
\begin{align*}
\int_{-\infty}^{\infty}&\frac {e^{-\xi^2}}{u+i(\xi-(t-s)/2)}\, d\xi\\
&=i\int_{-\infty}^{\infty}\frac {e^{-\xi^2}}{(t-s)/2+iu-\xi}\, d\xi\\
&=-\pi e^{-((t-s)/2+iu)^2}\erfc\left[i\left((t-s)/2-u\right)\right],
\end{align*}
where we have used Faddeeva's formula.
It follows that
\begin{equation}\label{FT:5}K=-\int_{-\infty}^{-x/2}e^{2isu}\erfc(-u+i(t-s)/2)\, du,\quad (x\ge 0).\end{equation}
For $x\le 0$ we have instead
\begin{equation*}
K=-\int_{-\infty}^0e^{2isu}\erfc(-u+i(t-s)/2)\, du+
\int_{0}^{-x/2}e^{2isu}\erfc(u+i(s-t)/2)\, du,\quad (x\le 0).\end{equation*}

Now observe that, for $x\ge 0$, \eqref{FT:5} implies
\begin{align*}-K(x,t,s)&=\int_{-\infty}^{-x/2}e^{2isu}\erfc(-u+i(t-s)/2)\, du\\
&=\left[\frac {e^{2isu}}{2is}\erfc(-u+i(t-s)/2)\right]^{u=-x/2}_{u=-\infty}-\int_{-\infty}^{-x/2}\frac {e^{2isu}}
{2is}\frac 2 {\sqrt{\pi}}e^{-(u+i(s-t)/2)^2}\, du\\
&=\frac {e^{-isx}}{2is}\erfc\left(\frac {x+it-is} 2\right)-\frac 1 {is\sqrt{\pi}}e^{-st}\int_{-\infty}^{-x/2}
e^{-(u+i(s+t)/2)^2}\, du\\
&=\frac {e^{-isx}}{2is}\left(2-\erfc\left(\frac {x+is-it} 2\right)\right)-\frac
{e^{-st}}{2is}\erfc\left(\frac{x+it+is} 2\right)\\
&=\frac {e^{-isx}}{is}E(x-is+it)-\frac {e^{-st}}{is}E(x+it+is).
\end{align*}

We now assume that $x\le 0$ and write $\tilde{K}(x,s,t)=K(x,s,t)-K(0,s,t)$ so that
\begin{align*}\tilde{K}(x,s,t)&=\int_0^{-x/2}e^{2isu}\erfc(u+i(s-t)/2)\, du\\
&=\left[\frac {e^{2isu}}{2is}\erfc(u+i(s-t)/2)\right]_{u=0}^{u=-x/2}+
\int_0^{-x/2}\frac {e^{2isu}}{2is}\frac 2 {\sqrt{\pi}}e^{-(u+i(s-t)/2)^2}\, du\\
&=\frac {e^{-ixs}}{2is}\erfc\left(\frac {-x+is-it}2\right)-\frac 1 {2is}\erfc\left(\frac{is-it}2\right)\\
&+\frac {e^{-st}}{2is}\frac 2 {\sqrt{\pi}}\int_0^{-x/2}e^{-(u-i(s+t)/2)^2}\, du
\end{align*}

This means that
\begin{equation*}\tilde{K}(x,s,t)=\frac {e^{-ixs}}{is}(1-E(x+it-is))-\frac 1 {is}E(is-it)
+\frac {e^{-st}}{is}(E(x+it+is)-E(it+is)).\end{equation*}
The formula \eqref{FT:4} is immediate.
\end{proof}

Set
\begin{equation}\label{etau}E_\parat(z) := E(z-2\parat) - E(z+2\parat).\end{equation}
By the previous lemma, we have
\begin{align}\label{FT:3}
i&K_\parat(x,s,t)  = \frac {e^{-is(x-2\parat)}-1}s  + \frac {e^{-st}}{s} E_\parat(x+i(s+t))\\
&-\frac {e^{-is(x-2\parat)}}{s}E(x-2\parat+i(t-s)) +\frac {e^{-is(x+2\parat)}}{s}E(x+2\parat+i(t-s)) \nonumber
\end{align}
for $x\in(-2\parat,2\parat).$

\subsection{Translation invariant solutions to Ward's equation} We now prove Theorem \ref{THV}.

Let $g$ and $\fii$ be unknown functions (say in $L^\infty(\R)$)
and put
$$\Phi:=\gamma*\fii,\qquad G:=\gamma*g.$$

Then Ward's equation \eqref{Ward:7} is equivalent to the following system:
\begin{align}\label{Ward:8:1}L&=G\Phi-\Phi',\quad \text{and}\\
\label{Ward:8:2}\fii&=g'+1.\end{align}
We shall refer to the equations \eqref{Ward:8:1} and \eqref{Ward:8:2} as Ward's \textit{first} and \textit{second} equation, respectively.

Since $\Phi=G'+1$ we have
$$\hat{\Phi}(s)=2\pi\cdot \delta(s)+is\hat{G}(s),$$
where $\delta$ is Dirac measure at $0$.
Moreover, by Fourier's inversion formula, the function $L$ in \eqref{ldef} satisfies
\begin{equation}\label{FT:8}L(x)=\frac i{(2\pi)^2} \iint_{\R^2} e^{ix(s+t)}(iK_\parat)(x,s,t)\hat{\Phi}(s)\hat{\Phi}(t)\, dsdt.\end{equation}
It follows that
$$L=L_1+L_2,$$ where
\begin{align}\label{def:L1}L_1(x)&:=\frac i{2\pi}\int e^{ixt}(iK_\parat)(x,t,0)\hat{\Phi}(t)\, dt,\\
\label{l2}L_2(x)&:=-\frac 1 {(2\pi)^2}\iint se^{ix(s+t)}(iK_\parat)(x,s,t)\hat{G}(s)\hat{\Phi}(t)\, dsdt.\end{align}
In order to analyze this decomposition, we shall prove a few lemmas. In the sequel, we denote by
$\mu$ the operation of multiplication by the dependent variable,
\begin{equation}\label{def:mu}\left[\mu f\right](x):=x\cdot f(x).\end{equation}

\begin{lem} \label{Gauss:rep:2} Suppose that $\Phi=\gamma *\fii$.
Then
$$\mu \Phi=\gamma*\left[\mu \fii\right]-\Phi'.$$
\end{lem}

\begin{proof} Since $\gamma'(x)=-x\gamma(x)$ we have
$$\Phi'(x)=\int_\R \gamma'(x-t)\fii(t)\, dt=\int_\R(t-x)\gamma(t-x)\fii(t)\, dt=\gamma*\left[\mu \fii\right](x)-x\Phi(x).$$
The proof of the lemma is complete.
\end{proof}

The following lemma uses the plasma functions $E$ and $F$ of Section \ref{plasma:functions} as well as the functions
\begin{equation}\label{def:aA}a(x):=xF(x)-\gamma(x),\qquad A:=\gamma *a.\end{equation}

\begin{lem} We have $a'=F$ and $A(x)=xE(x)+2E'(x)$.
\end{lem}

\begin{proof} It is clear that $a'=F$. For the other statement we
 shall first prove that
\begin{equation}\label{FT:9}\mu E=\gamma*\left[\mu F+\gamma\right].\end{equation}
Indeed, since $E=\gamma *F$,
$$\calF\left[\mu E\right](\xi)=i\hat{E}'(\xi)=i\left(\hat{\gamma}\hat{F}\right)'(\xi)=
-i\xi\hat{\gamma}(\xi)\hat{F}(\xi)+i\hat{\gamma}(\xi)\hat{F}'(\xi),$$
so
$$\calF\left[\mu E\right](\xi)/\hat{\gamma}(\xi)=-i\xi\hat{F}(\xi)+i\hat{F}'(\xi)=\calF\left[-F'(x)+xF(x)\right](\xi),$$
establishing \eqref{FT:9}.
But $E'=-\gamma *\gamma$, so by \eqref{FT:9}, we get
$$\mu E+2E'=\gamma*\left[\mu F+\gamma\right]-2\gamma *\gamma=\gamma* a.$$
The proof of the lemma is complete.
\end{proof}

It is clear from the relation \eqref{FT:4} that the entire function $A(z)=zE(z)+2E'(z)$ satisfies
\begin{equation}(iK_\parat)(x,0,t)=-i(x-2\parat)+iA_\parat(x+it),\end{equation}
where $A_\parat(z) = A(z-2\parat) - A(z+2\parat).$
From this, we see that the function $L_1$ in \eqref{def:L1} takes the form
\begin{align*}L_1(x)=\frac i{2\pi}\int_{-\infty}^\infty e^{ixt}\left(-i(x-2\parat)+iA_\parat(x+it)\right)\hat{\Phi}(t)\, dt.\end{align*}
We now define
\begin{equation}\label{def:M}M_\parat(x)=\frac 1 {2\pi}\int_{-\infty}^\infty e^{ixt}A_\parat(x+it)\hat{\Phi}(t)\, dt,\end{equation}
and note that
\begin{equation}\label{FT:10}L_1(x)=(x-2\parat)\Phi(x)-M_\parat(x).\end{equation}
Set
$$a_\parat(z) = a(z-2\parat) - a(z+2\parat),$$
where $a$ is given by \eqref{def:aA}.
\begin{lem} \label{FT:11} We have $M_\parat=\gamma *(\fii a_\parat).$
\end{lem}

\begin{proof} Using $A_\parat=\gamma*a_\parat$ and $\Phi=\gamma*\fii$ we compute
\begin{align*}2\pi M_\parat(x)&=\int_\R e^{ixt}\hat{\Phi}(t)\, dt\int_\R\gamma(x+it-u)a_\parat(u)\, du\\
&=\sqrt{2\pi}\int_\R \gamma(u)e^{ux}a_\parat(u)\, du\int_\R e^{ixt}\gamma(x+it)e^{iut}\hat{\Phi}(t)\, dt\\
&=\sqrt{2\pi}\,\gamma(x)\int_\R \gamma(u)e^{ux}a_\parat(u)\, du\int_\R e^{t^2/2+iut}\hat{\Phi}(t)\, dt\\
&=\sqrt{2\pi}\,\gamma(x)\int_\R \gamma(u)e^{ux}a_\parat(u)\, du\int_\R e^{iut}\hat{\fii}(t)\, dt\\
&=\sqrt{2\pi}\,\gamma(x)\cdot 2\pi\int_\R \gamma(u)e^{ux}a_\parat(u)\fii(u)\, du\\
&=2\pi\int_\R \gamma(x-u)a_\parat(u)\fii(u)=2\pi(\gamma *(\fii a_\parat ))(x).
\end{align*}
The proof of the lemma is complete.
\end{proof}

We finally define two auxiliary functions $N_\parat$ and $P_\parat$ by
\begin{align}\label{def:N}N_\parat(x)&=\frac 1 {(2\pi)^2}\iint_{\R^2} e^{itx} e^{2is\parat} E(x-2\parat + i(t-s))  \hat{\Phi}(t)\hat{G}(s)\, dsdt\\ \nonumber
&-\frac 1 {(2\pi)^2}\iint_{\R^2} e^{itx} e^{-2is\parat} E(x+2\parat + i(t-s))  \hat{\Phi}(t)\hat{G}(s)\, dsdt
;\\
\label{def:P}P_\parat(x)&=\frac 1 {(2\pi)^2}\iint_{\R^2} e^{ix(s+t)}e^{-st}E_\parat(x+it+is)\hat{\Phi}(t)\hat{G}(s)\, dsdt.
\end{align}
In view of the relation \eqref{FT:3}, we have
\begin{align*}L_2(x)&=-G(2\parat)\Phi(x)+G(x)\Phi(x)+N_\parat(x)-P_\parat(x).
\end{align*}

Now recall that Ward's first equation \eqref{Ward:8:1} takes the form $L= L_1+L_2=G\Phi-\Phi'$. By the formula
\eqref{FT:10} for $L_1$ and the above expression for $L_2$, Ward's first equation
is equivalent to
$$G\Phi-\Phi'=(\mu-2\parat) \Phi-M_\parat+G\Phi-G(2\parat)\cdot \Phi+N_\parat-P_\parat,\qquad (\left[\mu \Phi\right](x)=x\Phi(x)).$$
The last equation transforms to
\begin{equation}\label{Ward:9:1}\Phi'+\mu \Phi+c\Phi=M_\parat+P_\parat-N_\parat,\quad (c=-G(2\parat)-2\parat).\end{equation}
Recalling that
$\Phi=\gamma*\fii$, $G=\gamma*g$,
and using Lemma \ref{Gauss:rep:2}, we obtain the following result.
\begin{lem} \label{Ward:10:1} Ward's first equation \eqref{Ward:9:1} can be written
$$\gamma*\left[\mu \fii\right]+c\cdot \gamma *\fii=\gamma*\left[m_\parat+p_\parat-n_\parat\right]$$
where $M_\parat=\gamma *m_\parat$, $P_\parat=\gamma*p_\parat$, $N_\parat=\gamma*n_\parat$, and $c=-G(2\parat)-2\parat$.
\end{lem}

In order to apply the lemma, we need to solve the equations $P_\parat=\gamma*p_\parat$ and $N_\parat=\gamma*n_\parat$.
This is done in the next two lemmas.

\begin{lem} The function $P_\parat$ in \eqref{def:P} satisfies $P_\parat=\gamma*p_\parat$ where $$p_\parat=g\fii F_\parat, \quad F_\parat(z) = F(z-2\parat) - F(z+2\parat).$$
\end{lem}

\begin{proof} Write $a=s+t$ and
$$E(x+ia)=\frac 1 {2\pi}\int_\R e^{ixu}e^{-au}\hat{E}(u)\, du.$$
Then
\begin{align*}P_\parat(x)&=\frac 1 {(2\pi)^3}\iiint_{\R^3} e^{ixa}e^{-st}
\big(e^{i(x-2\parat)u} -e^{i(x+2\parat)u} \big)e^{-au}\hat{E}(u)\hat{\Phi}(t)\hat{G}(s)\, dudsdt\\
&=\frac 1 {(2\pi)^3}\iiint e^{ix(a+u)}e^{-st-au}  \big(e^{-2i\parat u} -e^{2i\parat u} \big) \hat{E}(u)
\hat{\Phi}(t)\hat{G}(s)\, dudsdt.
\end{align*}
Taking Fourier transform with respect to $x$ gives
\begin{align*}\hat{P}(\xi)&=\frac 1 {(2\pi)^2}\iiint \delta_\xi(a+u) e^{-st-au}\big(e^{-2i\parat u} -e^{2i\parat u} \big) \hat{E}(u)\hat{\Phi}(t)\hat{G}(s)\,dudsdt\\
&=\frac 1 {(2\pi)^2}\iint e^{-st-a(\xi-a)}\big(e^{-2i\parat (\xi-a)} -e^{2i\parat (\xi-a)} \big) \hat{E}(\xi-a)\hat{\Phi}(t)\hat{G}(s)\, dsdt.
\end{align*}
Since $E=\gamma*F$, the last expression equals
\begin{align*}\frac 1 {(2\pi)^2}\iint e^{-st-a(\xi-a)}\big(e^{-2i\parat (\xi-a)} -e^{2i\parat (\xi-a)} \big)\hat{\gamma}(\xi-a)\hat{F}(\xi-a)\hat{\gamma}(t)\hat{\fii}(t)\hat{\gamma}(s)\hat{g}(s)\, dsdt.
\end{align*}
But $\hat{\gamma}(\xi-a)=\hat{\gamma}(\xi)\hat{\gamma}(a)e^{a\xi}$ so the integrand is
$$e^{-st+a^2}\hat{\gamma}(\xi)\hat{\gamma}(a)\hat{F}_\parat(\xi-a)\hat{\gamma}(t)\hat{\fii}(t)\hat{\gamma}(s)\hat{g}(s)
=e^{-st}\hat{\gamma}(a)^{-1}\hat{\gamma}(\xi)\hat{F}_\parat(\xi-a)\hat{\gamma}(t)\hat{\fii}(t)\hat{\gamma}(s)\hat{g}(s),$$
and since $a=s+t$ and $e^{-st}\hat{\gamma}(s)\hat{\gamma}(t)=\hat{\gamma}(a)$ this simplifies to
$$\hat{\gamma}(\xi)\hat{F}_\parat(\xi-s-t)\hat{\fii}(t)\hat{g}(s)$$
We have shown that
$$\frac{\hat{P}_\parat(\xi)}{\hat{\gamma}(\xi)}=\frac 1 {(2\pi)^2}\iint \hat{F}_\parat(\xi-s-t)\hat{\fii}(t)\hat{g}(s)\, dsdt=\frac 1 {(2\pi)^2}\hat{F}_\parat*\hat{g}*\hat{\fii}(\xi).$$
Taking inverse Fourier transforms finishes the proof of the lemma.
\end{proof}

Let $\gamma_\parat(x) := \gamma(x-2\parat) - \gamma(x+2\parat).$

\begin{lem} The function $N_\parat$ in \eqref{def:N} satisfies $N_\parat=\gamma *n_\parat$ where
$$n_\parat=\fii \left[\1_{(-\infty,0)}*(\gamma_\parat g)\right].$$
\end{lem}

\begin{proof} By Fourier's inversion formula, we can write
$$N_\parat(x)=\frac 1 {(2\pi)^3}\iiint_{\R^3} e^{itx+iu(x+it-is)} \big(e^{2i(s-u)\parat}-e^{-2i(s-u)\parat}\big) \hat{E}(u)\hat{\Phi}(t)\hat{G}(s)\, dudtds.$$
Taking Fourier transform in $x$, using that $E=\gamma*F$, we get
\begin{align*}\hat{N}_\parat(\xi)&=
\frac 1 {(2\pi)^2}\iiint \delta_\xi(t+u)e^{u(s-t)} \big(e^{2i(s-u)\parat}-e^{-2i(s-u)\parat}\big) \hat{E}(u)\hat{\Phi}(t)\hat{G}(s)\, dudtds\\
&=\frac 1 {(2\pi)^2}\iint e^{(\xi-t)(s-t)} \big(e^{-2i(\xi-s-t)\parat}-e^{2i(\xi-s-t)\parat}\big) \hat{E}(\xi-t)\hat{\Phi}(t)\hat{G}(s)\, dtds.
\end{align*}
Since $\hat{E}(\xi-t)\hat{\Phi}(t)\hat{G}(s) = \hat{\gamma}(\xi)\hat{\gamma}(t)e^{\xi t}
\hat{F}(\xi-t)\hat{\gamma}(t)\hat{\fii}(t)\hat{G}(s),$
the last equation simplifies to
\begin{align*}
\frac{\hat{N}_\parat(\xi)}{{\hat{\gamma}(\xi)}}  = \frac1{(2\pi)^2}\iint e^{\xi s-st}  \big(e^{-2i(\xi-s-t)\parat}-e^{2i(\xi-s-t)\parat}\big) \hat{F}(\xi-t)\hat{\fii}(t)\hat{G}(s)\, dtds.
\end{align*}
We now make the observation that
\begin{align*}\frac 1 {2\pi}&\int e^{i\xi x} e^{\xi s} \big(e^{-2i(\xi-s-t)\parat}-e^{2i(\xi-s-t)\parat}\big)\hat{F}(\xi-t)\, d\xi \\
&=e^{it(x-is)}\frac 1 {2\pi}\int e^{iv(x-is)}\big(e^{-2i(v-s)\parat}-e^{2i(v-s)\parat}\big)\hat{F}(v)\, dv\\
&=e^{it(x-is)}\big(e^{2is\parat}F(x-2\parat-is)-e^{-2is\parat}F(x+2\parat-is)\big).\end{align*}
Hence,  defining $n_\parat(x)$ by
$\hat{n}_\parat(\xi)=\hat{N}_\parat(\xi)/\hat{\gamma}(\xi)$ and applying the inverse Fourier transform to $\hat{n}_\parat,$ we
find (since $\hat{\gamma}(t)=e^{-t^2/2}$ and $F=\1_{(-\infty,0)}*\gamma$)
\begin{align*}
2\pi n(x)&=\frac 1 {2\pi}\iint_{\R^2} e^{itx}\big(e^{2is\parat}F(x-2\parat-is)-e^{-2is\parat}F(x+2\parat-is)\big)\hat{\fii}(t)\hat{G}(s)\, dtds\\
&=\frac 1 {(2\pi)^{3/2}}\iint_{\R^2} e^{itx}\left[\int_{-\infty}^{\,0} e^{2is\parat} e^{-(x-2\parat-is+u)^2/2-s^2/2}\, du\right]
\hat{\fii}(t)\hat{g}(s)\, dtds\\
&-\frac 1 {(2\pi)^{3/2}}\iint_{\R^2} e^{itx}\left[\int_{-\infty}^{\,0}  e^{-2is\parat} e^{-(x+2\parat-is+u)^2/2-s^2/2}\, du\right]
\hat{\fii}(t)\hat{g}(s)\, dtds\\
&=\frac 1 {(2\pi)^{3/2}}\int_\R e^{itx}\hat{\fii}(t)\,dt\int_{-\infty}^{\,0} \big(e^{-(x+u-2\parat)^2/2}-e^{-(x+u+2\parat)^2/2} \big)  \int_\R e^{is(x+u)}\hat{g}(s)\, dsdu
\\
&=\int_{-\infty}^\infty e^{itx}\hat{\fii}(t)\,dt \frac 1 {\sqrt{2\pi}}\int_{-\infty}^{\,0} \big(e^{-(x+u-2\parat)^2/2}-e^{-(x+u+2\parat)^2/2} \big) g(x+u)\, du\\
&=2\pi  \fii(x) \1_{(-\infty,0)}*(\gamma_\parat g)(x).
\end{align*}
The proof of the lemma is complete.
\end{proof}

We now appeal to Ward's first equation (Lemma \ref{Ward:10:1})
$$\gamma*(\mu \fii+c\fii)=\gamma*(m_\parat+p_\parat-n_\parat),\qquad (c=-G(2\parat)-2\parat).$$
This is equivalent to
$$(x+c)\fii(x)=m_\parat(x)+p_\parat(x)-n_\parat(x),\quad (\text{for\, a.e.\,\,} x).$$
But by the preceding computations, $m_\parat=\fii a_\parat,$ $p=g\fii F_\parat$, and $n= \fii \left[\1_{(-\infty,0)}*(\gamma_\parat g)\right]$ so we obtain the equivalent equation
\begin{equation}\label{FT:13}x+c=a_\parat(x)+g(x)F_\parat(x)-\left[\1_{(-\infty,0)}*(\gamma_\parat g)\right](x)
\quad\text{when}\quad \fii(x)\ne 0.\end{equation}

Before proceeding, note that since the distributional derivative
$\1_{(-\infty,0)}'=-\delta$, we have
\begin{equation*}\left[\1_{(-\infty,0)} *(\gamma_\parat g)\right]'=-\gamma_\parat g.\end{equation*}
Thus differentiating in \eqref{FT:13}, recalling that $a_\parat'=F_\parat$, we arrive at the equation
$$1=F_\parat+(gF_\parat)'+\gamma_\parat g.$$
Since $g'=\fii-1$ and $F_\parat'=-\gamma_\parat$ we finally arrive at
$$1=F_\parat+(\fii-1)F_\parat-\gamma_\parat g+\gamma_\parat g=\fii F_\parat,\qquad \text{a.e. \, on}\quad \{\fii\ne 0\}.$$
This means that $\fii=1/F_\parat$ whenever $\fii\ne 0$, so Ward's first equation is equivalent to that (almost everywhere)
$$\fii=\frac {\1_e} {F_\parat}$$
where $e$ is a Borel subset of $\R$. We can here clearly assume that $e$ be closed.

We now claim that $e$ is some interval of positive measure.
To show this, we first rewrite \eqref{FT:13} as
$$x + C = a_\parat(x)+g(x)F_\parat(x) + \int_0^x \gamma_\parat(t) g(t)\,dt
\quad\text{on}\quad e.$$
Here $C$ is some constant.
By means of integration by part, up to an additive constant
$$x = a_\parat(x)+\int_0^x F_\parat(t) g'(t)\,dt
\quad\text{on}\quad e.$$
Since $g'=\fii-1$ and $a_\parat'=F_\parat,$
$$x = \int_0^x F_\parat(t) \fii(t)\,dt = \int_0^x \1_e(t)\,dt \quad\text{on}\quad e.$$
Thus $e$ is connected.

We have proved that Ward's equation is satisfied if and only if
\begin{equation*}\Phi(z)=\gamma*\frac {\1_I} {F_\parat}(z)=\frac 1 {\sqrt{2\pi}}\int_I \frac {e^{-(z-t)^2/2}}{F_\parat(t)}\, dt,\end{equation*}
where $I$ is some interval of positive measure.

The proof of Theorem \ref{THV} is finished.
q.e.d.

\subsection{The mass-one theorem}

We now finally prove the mass-one theorem (Theorem \ref{THVI}).

Suppose that $\Phi=\gamma*\fii$ is an error-type function satisfying the mass-one equation in $\Sigma_\parat$, i.e.,
\begin{equation}\label{moq1}\begin{split}\Phi(x)&=\int_{\Sigma_\parat}e^{-|x/2-w|^2}|\Phi(x/2+w)|^2\, dA(w)\\
&=
\int_{\Sigma_\parat-x/2}e^{-|w|^2}\babs{\,\Phi(w+x)\,}^{\,2}\, dA(w),\qquad (-2\parat<x<2\parat).\\
\end{split}
\end{equation}

Consider the Fourier transform $\hat{\Phi}=\hat{\gamma}\hat{\fii}$ (as in Section
\ref{Fourier:sec}) and apply Fourier's inversion formula:
$$\Phi(x+u)=\frac 1 {2\pi}\int_\R e^{is(x+u)}\hat{\Phi}(s)\, ds.$$
The equation \eqref{moq1} then becomes
\begin{equation}\label{ttf}
\Phi(x)=\frac 1 {(2\pi)^2}\iint_{\R^2}e^{i(t+s)x}\hat{\Phi}(s)\hat{\Phi}(t)\, dsdt\int_{\Sigma_\parat-x/2}
e^{-\babs{w}^2}e^{iwt}e^{i\bar{w}s}\, dA(w).
\end{equation}

\begin{lem} \label{l1} We have
$$\int_{\Sigma_\parat-x/2}e^{-|w|^2}e^{iwt}e^{i\bar{w}s}\, dA(w)=e^{-st}E_\parat(x+it+is).$$
\end{lem}

(Here $E_\parat(z)=E(z-2\parat)-E(z+2\parat)$, see \eqref{etau}.)

\begin{proof} We shall first compute the integral
$$I(\parat)=\int_{\L+\parat-x/2}e^{-|w|^2}e^{iwt}e^{i\bar{w}s}\, dA(w)$$
where $\L=\{z;\, \re z<0\}$.
We obtain
\begin{align*}I(\parat)&=\frac 1 \pi\int_{-\infty}^{\parat-x/2}e^{-u^2+iu(t+s)}\, du\int_{-\infty}^{+\infty}
e^{-v^2-v(t-s)}\, dv\\
&=\frac 1 \pi e^{-(t+s)^2/4}\int_{-\infty}^{\parat-x/2}e^{-(u-i(t+s)/2)^2}\, du
\,e^{(t-s)^2/4}\int_{-\infty}^{+\infty}e^{-(v+(t-s)/2)^2}\, dv\\
&=\frac 1 {\sqrt{\pi}}e^{-(s^2+t^2)/2}\int_{-\infty}^{\parat-x/2}e^{-(u-i(t+s)/2)^2}\, du\\
&=e^{-st}\frac 1 {\sqrt{2\pi}}\int_{-\infty}^{(-2\parat+x+i(t+s))/\sqrt{2}}e^{-z^2/2}\, dz=e^{-st}E(-2\parat+x+is+it).
\end{align*}
Finally,
$$\int_{\Sigma_\parat-x/2}e^{-|w|^2}e^{iwt}e^{i\bar{w}s}\, dA(w)=I(\parat)-I(-\parat).$$
The proof of the lemma is finished. \end{proof}

It follows from the lemma that the mass-one equation is equivalent to that
\begin{equation}\label{tts}\Phi(x)=\frac 1 {(2\pi)^2}\iint_{\R^2}e^{ix(t+s)}e^{-st}E_\parat(x+it+is)
\hat{\Phi}(t)\hat{\Phi}(s)\, dtds.\end{equation}

\begin{lem} The function $\Phi=\gamma*\fii$ satisfies the mass-one equation if and
only if $\fii=\fii^2F$.
\end{lem}

\begin{proof} If $a=s+t$, then
$$E_\parat(x+ia)=\frac 1 {2\pi}\int_\R e^{ixu}e^{-au}\hat{E}_\parat(u)\, du,$$
so the mass-one equation \eqref{tts} means that
$$\Phi(x)=\frac 1 {(2\pi)^3}\iiint_{\R^3}e^{ix(a+u)-au-st}\hat{E}_\parat(u)\hat{\Phi}(s)\hat{\Phi}(t)\, dudsdt.$$
Taking the Fourier transform with respect to $x$, we obtain the equivalent equation
\begin{align*}\hat{\Phi}(\xi)&=\frac 1 {(2\pi)^2}\iiint_{\R^3} \delta_\xi(a+u)e^{-au}\hat{E}_\parat(u)\, du\, e^{-st}\, dsdt\\
&=\frac 1 {(2\pi)^2}\iint_{\R^2} e^{-a(\xi-a)}\hat{E}_\parat(\xi-a)e^{-st}\hat{\Phi}(s)\hat{\Phi}(t)\, dsdt.
\end{align*}
Recalling that $E_\parat=\gamma *F_\parat$ and $\Phi=\gamma*\fii$ this transforms to
\begin{equation}\label{prel}\hat{\fii}(\xi)\hat{\gamma}(\xi)=\frac 1 {(2\pi)^2}\iint_{\R^2}e^{-a(\xi-a)}\hat{\gamma}(\xi-a)\hat{F}_\parat(\xi-a)e^{-st}
\hat{\gamma}(s)\hat{\fii}(s)
\hat{\gamma}(t)\hat{\fii}(t)\, dsdt.\end{equation}
Using that $\hat{\gamma}(\xi-a)=\hat{\gamma}(\xi)\hat{\gamma}(a)e^{a\xi}$ and $\hat{\gamma}(a)=\hat{\gamma}(s)\hat{\gamma}(t)e^{-st}$ (since $a=s+t$), we find
\begin{align*}e^{-a(\xi-a)}\hat{\gamma}(\xi-a)e^{-st}\hat{\gamma}(s)\hat{\gamma}(t)=e^{a^2}\hat{\gamma}(\xi)
\hat{\gamma}(a)^2=\hat{\gamma}(\xi),\end{align*}
so \eqref{prel} is equivalent to that
$$\hat{\fii}(\xi)=\frac 1 {(2\pi)^2}\iint_{\R^2} \hat{F}_\parat(\xi-s-t)\hat{\fii}(s)\hat{\fii}(t)\, dsdt=\frac 1 {(2\pi)^2} \hat{F}_\parat*\hat{\fii}*\hat{\fii}(\xi).$$
Taking inverse Fourier transforms, this gives $\fii=\fii^2 F_\parat$.
\end{proof}

\begin{proof}[Proof of Theorem \ref{THVI}] A real-valued Borel function $\fii$ satisfies $\fii=\fii^2 F_\parat$ if and only if $\fii=\1_e/F_\parat$ for a Borel set $e\subset\R$.
Hence Theorem \ref{THVI} is a consequence of the preceding lemma.
\end{proof}

\subsection{The case of a regular boundary point.} \label{ssu} If in our above proofs of Theorem \ref{THV} and Theorem \ref{THVI} we replace the functions $E_\parat$ and $F_\parat$ by $E$ and $F$ respectively,
we obtain a proof
of Theorem \ref{TH8}. q.e.d.

\end{document}